\UseRawInputEncoding
\documentclass[final]{siamltex}

\usepackage[leqno]{amsmath}
\usepackage{amssymb}
\usepackage{framed}
\usepackage{appendix}
\usepackage{epstopdf}
\usepackage{framed}
\usepackage{bm}
\newcommand{\rmnum}[1]{\romannumeral #1}

\newtheorem{algo}{Algorithm}
\newtheorem{rem}{Remark}

\title{Lagrangian-based methods in convex optimization: prediction-correction frameworks with non-ergodic  convergence rates\thanks{This work was supported by the National Natural Science Foundation of China under Grant 12171021 and
the Fundamental Research Funds for the Central Universities.}}
\author{Tao Zhang\footnotemark[2] \and Yong Xia\footnotemark[2] \and Shiru Li\footnotemark[2]}

\begin{document}
\maketitle

\renewcommand{\thefootnote}{\fnsymbol{footnote}}
\footnotetext[2]{School of Mathematical Sciences, Beihang University, Beijing, 100191, P. R. China  ({\tt (T. Zhang) shuxuekuangwu@buaa.edu.cn};  {\tt (Y. Xia) yxia@buaa.edu.cn}; {\tt (S. Li, Corresponding author) lishiru@buaa.edu.cn}).}

\begin{abstract}
Lagrangian-based methods are classical methods for solving convex optimization problems with equality constraints. We present novel prediction-correction frameworks for such methods and their variants, which can achieve $O(1/k)$ non-ergodic convergence rates for general  convex optimization and $O(1/k^2)$ non-ergodic convergence rates  under the assumption that the objective function is strongly convex or gradient Lipschitz continuous. We give two approaches ($updating~multiplier~once$ $or~twice$) to design algorithms satisfying the presented prediction-correction frameworks. As applications,  we establish  non-ergodic convergence rates for some well-known Lagrangian-based methods  (esp., the ADMM type methods and the multi-block ADMM type methods).
\end{abstract}



\pagestyle{myheadings}
\thispagestyle{plain}
\markboth{T ZHANG,  Y XIA,  AND S LI}{NON-ERGODIC CONVERGENCE RATES}

\section{Introduction}
The convex programming problem with linear equality constraints is a widely applied class of optimization problems. The mathematical  formulation reads as
\begin{equation}\label{P1}\tag{P1}
\min\{ f(x):~ Ax=b\},
\end{equation}
where $f:\mathbb{R}^n\rightarrow\mathbb{R}$ is closed, proper, convex, but not necessarily smooth, $A\in\mathbb{R}^{l\times n}$, and $b\in\mathbb{R}^{l}$. The feasible region of \eqref{P1} is denoted by $\Omega$. An important special case of \eqref{P1} is
the following two-block separable convex optimization problem:
\begin{equation}\label{P2}\tag{P2}
\min \left\{f(x)=f_1(x_1)+f_2(x_2):~ (Ax:=)A_1x_1+A_2x_2=b\right\},
\end{equation}
where $A_1\in\mathbb{R}^{l\times n_1}$, $A_2\in\mathbb{R}^{l\times n_2}$, $n_1+n_2=n$, $b\in\mathbb{R}^{l}$, $f_1:\mathbb{R}^{n_1}\rightarrow\mathbb{R}$ and $f_2:\mathbb{R}^{n_2}\rightarrow\mathbb{R}$ are proper, closed and convex. Then,
a natural extension of \eqref{P2} leads to the multi-block separable convex optimization problem:
\begin{equation}\label{P3}\tag{P3}
\min\left\{f(x)=\sum_{i=1}^{m}f_i(x_i):~ (Ax:=)\sum_{i=1}^{m}A_ix_i=b\right\},
\end{equation}
where $m\geq2$, $f_i:\mathbb{R}^{n_i}\rightarrow\mathbb{R}$ is closed proper convex for $i\in[1,2,\cdots,m]$, $A_i\in\mathbb{R}^{l\times n_i}$, $\sum_{i=1}^{m}n_i=n$ and $b\in\mathbb{R}^{l}$,  which remains a special case of \eqref{P1}.

The augmented Lagrangian method (ALM) \cite{hestenes1969multiplier,powell1969method} is a fundamental and efficient approach for solving problem \eqref{P1}. Several improved variants of ALM have been developed that are highly effective. For instance, the proximal ALM, introduced by Rockafeller \cite{rockafellar1976augmented,rockafellar1976monotone}, and the linearized ALM, which employs special proximal terms, can easily solve the subproblem of ALM.
Another well-known Lagrangian-based method, the alternating direction method of multipliers (ADMM) \cite{gabay1976dual,glowinski1975}, efficiently solves problem \eqref{P2} by taking advantage of its special structure. ADMM is a very popular method, with many improved variants including the proximal ADMM \cite{eckstein1994some} and the linearized ADMM \cite{yang2013linearized} with special proximal terms.
Directly applying the idea of ADMM to solve problem \eqref{P3} does not lead to convergence, but several ADMM type methods \cite{he2012alternating,he2017convergence} have been presented to solve problem \eqref{P3} effectively.

The purpose of this paper is to study  the non-ergodic convergence rates of the Lagrangian-based methods for solving \eqref{P1}, \eqref{P2} and \eqref{P3}, respectively.

He and Yuan \cite{2012On} presented a unified prediction-correction framework for simplifying the convergence and convergence rate proofs of Lagrangian-type methods. Algorithms that satisfy this framework demonstrate $O(1/K)$\footnote{Throughout this paper,
 we use $k$ and $K$ in describing the non-ergodic and ergodic convergence rates, respectively.} convergence rates of the primal-dual gap in both ergodic and non-ergodic senses \cite{he2018my,he2014strictly,he2016convergence,he2020optimally,he2015splitting}. In particular, the works of Chambolle and Pock \cite{chambolle2011first,chambolle2016ergodic}, Monteiro and Svaiter \cite{monteiro2013iteration}, Shefi and Teboulle \cite{shefi2014rate}, and He and Yuan \cite{he20121} show that the $O(1/K)$ convergence rate can be achieved in the ergodic sense.
 Furthermore, Chambolle and Pock \cite{chambolle2011first,chambolle2016ergodic} presented the primal-dual hybrid gradient method with $O(1/K^2)$ ergodic convergence rate for solving the min-max problem under the  strongly convex assumption. The ADMM presented by Xu  \cite{xu2017accelerated} and the linearized ADMM due to Ouyang et al.  \cite{ouyang2015accelerated} enjoy the same convergence rate under the same assumption of strongly convex. Tian and Yuan \cite{2016An} established $O(1/K^2)$ ergodic convergence rate of ADMM when only the gradient is assumed to be Lipschitz continuous.
Recently, our work \cite{zhang2022faster} present a generalized
prediction-correction framework to establish $O(1/K^2)$ ergodic convergence rates for some Lagrangian-based
methods.

Historically, the first accelerated gradient method with $O(1/k^2)$
convergence rate for unconstrained convex optimization was due to Nesterov \cite{nesterov1983method}. It  has been further generalized  to composite convex optimization with the simple proximal operator \cite{beck2009fast,tseng2010}.
He and Yuan \cite{0On} introduced Nesterov's momentum approach in  dual updating and obtained $O(1/k^2)$ non-ergodic convergence rate (in view of the dual objective residual) for solving \eqref{P1}. Goldstein et al. \cite{2014fast} extended this approach to solve \eqref{P2} by a fast version of ADMM under some strict conditions. Tran-Dinh and Zhu \cite{2020NON} and Valkonen \cite{valkonen2020inertial} studied the acceleration algorithms for the special case of \eqref{P2} with $A_2=-I$ and $b=0$.
Recently, Sabach and Teboulle \cite{sabach2022faster} presented a class of Lagrangian-based methods with faster convergence rates by  introducing the so-called $nice ~primal ~algorithmic~ map$. All the above mentioned works \cite{sabach2022faster,2020NON,valkonen2020inertial} established $O(1/k)$  and $O(1/k^2)$ non-ergodic convergence rates under the strongly convex assumption.  These convergence rates can also be established by the penalty  methods  \cite{li2017convergence,tran2019proximal} or the smoothing technique \cite{train}.

{\bf Contributions.}
We present prediction-correction frameworks to achieve $O(1/k)$ non-ergodic convergence rates for the general convex optimization and $O(1/k^2) $ non-ergodic convergence rates under the assumption that the objective function is either strongly convex or gradient Lipschitz continuous. The basic idea is to combine minimizing  the primal variables  of the augmented Lagrangian function by the accelerated gradient methods due to Nesterov  \cite{nesterov1983method}  with the dual updating in the prediction-correction framework.
All the non-ergodic convergence rates are built on  the convergence conditions \eqref{V9}-\eqref{V12} (Section \ref{S2}). These convergence conditions are precisely used to establish the ergodic convergence rate \cite{zhang2022faster}.
Based on our newly established prediction-correction frameworks,
we give two classes of Lagrangian-based methods named $correcting~multiplier$  $twice$ and $correcting~multiplier$ $once$ to achieve non-ergodic convergence rates for solving \eqref{P1}, \eqref{P2} and \eqref{P3}. In particular,
to the best of our knowledge, it is the first time to establish non-ergodic convergence rates for solving \eqref{P3} by the multi-block ADMM type methods. As a by-product,
we show that for solving  the strongly convex  case of \eqref{P1},  the best   residue in the first $k$ iterations  ($\min_{i}\|x^i-x^{i+1}\|^2$ ($i=0,\dots,k$)) converges at the rate of $O(1/k^4)$ (see Remark \ref{r4}). This is a novel observation compared with the ergodic case \cite{zhang2022faster}.

{\bf Outline.}
The rest of this paper is as follows. In Section 2, we present our previous generalized prediction-correction framework, which includes ergodic convergence rates and convergence conditions. Section 3 discusses the prediction-correction framework with $O(1/k)$ non-ergodic convergence rates. In Section 4, we present the prediction-correction framework with $O(1/k^2)$ non-ergodic convergence rates. Section 5 provides several algorithms for solving \eqref{P2} and \eqref{P3}.

{\bf Notation.} Let $x=(x_1,\cdots,x_m)\in \mathbb{R}^n$ be a column vector and $x_i$ be the $i$-th component  or subvector of $x$. The inner product of $x,y$ is denoted by $\langle x,y\rangle=x^Ty$.
$I_n\in\mathbb{R}^{n\times n}$ denotes the  identity matrix.
Define $\|x\|_D^2= x^TDx$  no matter whether $D$ is positive semidefinite. The Euclidean norm $\|x\|$ is $\|x\|_{I_n}$.
$\sigma_{\max}(D)$ and $\sigma_{\min}(D)$ denote the maximal  and minimal eigenvalues of $D$, respectively.
 $\partial f(x)$ represents the subdifferential  of the convex function $f(x)$.
$\nabla f(x)$ represents  the  gradient of the smooth function $f(x)$.
The following two definitions are standard.
\begin{definition}
	$f:\mathbb{R}^n\rightarrow\mathbb{R}$  is  $\sigma~(\geq0)$-strongly convex if there is  a constant $\sigma\geq0$ such that$$f(y)\geq f(x)+\langle f'(x),y-x\rangle+\frac{\sigma}{2}\|y-x\|^2,~f'(x)\in\partial f(x),~\forall x,y\in\mathbb{R}^n.
	$$
\end{definition}
\begin{definition}
	$f:\mathbb{R}^n\rightarrow\mathbb{R}$  is  $L$-gradient Lipschitz continuous  if $f$ is differentiable and there is a constant $L>0$ such that
	\begin{equation}
	\begin{aligned}
	&
	f(y)\leq f(x)+\langle \nabla f(x),y-x\rangle+\frac{L}{2}\|y-x\|^2,~\forall x,y\in\mathbb{R}^n.\\
	\Longleftrightarrow&
	f(y)\geq f(x)+\langle \nabla f(x),y-x\rangle+\frac{1}{2L}\|\nabla f(x)-\nabla f(y)\|^2,~ \forall x,y\in\mathbb{R}^n.
	\end{aligned}
	\end{equation}
\end{definition}
\section{Preparations}\label{S2}
We  write in the following the Lagrangian function of \eqref{P1}-\eqref{P3}:
$$L(x,\lambda)=f(x)-\lambda^T(Ax-b),$$
where $\lambda$ is the Lagrange multiplier. We call $(x^*,\lambda^*)$ a saddle point of $L(x,\lambda)$ if it holds that
$$
L(x^*,\lambda)\leq L(x^*,\lambda^*)\leq L(x,\lambda^*),~\forall x\in \mathbb{R}^{n},~\forall\lambda\in\mathbb{R}^{l}.
$$
Throughout this paper, for \eqref{P1}, $x\in\mathbb{R}^{n},x'\in\Omega~{\rm and}~\lambda\in\mathbb{R}^{l}$, we define
\begin{equation}
\label{V2}u=v=\begin{pmatrix}
x\\ \lambda
\end{pmatrix},u'=v'=\begin{pmatrix}
x'\\ \lambda
\end{pmatrix},u^*=v^*=\begin{pmatrix}
x^*\\ \lambda^*
\end{pmatrix}, F(u)=\begin{pmatrix}
-A^T\lambda\\Ax-b
\end{pmatrix}.
\end{equation}
For \eqref{P2}, $x_i\in\mathbb{R}^{n_i} $ ($i=1,2$) and $\lambda\in\mathbb{R}^{l}$, we define
\begin{equation}
\label{V18}
\begin{aligned}
&u=\begin{pmatrix}
x_1\\x_2\\\lambda
\end{pmatrix}, u'=\begin{pmatrix}
x_1'\\x_2'\\\lambda
\end{pmatrix},u^*=\begin{pmatrix}
x_1^*\\x_2^*\\\lambda^*
\end{pmatrix},~F(u)=\begin{pmatrix}
-A_1^T\lambda\\-A^T_2\lambda\\Ax-b
\end{pmatrix},\\&
v=\begin{pmatrix}
x_2\\\lambda
\end{pmatrix},v'=\begin{pmatrix}
x_2'\\\lambda
\end{pmatrix},v^*=\begin{pmatrix}
x_2^*\\\lambda^*
\end{pmatrix}, x'=\begin{pmatrix}
x_1'\\x_2'
\end{pmatrix}\in\Omega.
\end{aligned}
\end{equation}
For \eqref{P3},  $x_i\in\mathbb{R}^{n_i} $ ($i=1,\dots,m$) and $\lambda\in\mathbb{R}^{l}$, we define
\begin{equation}
\begin{aligned}
&\label{D13}
u=\begin{pmatrix}
x_1\\ \vdots\\x_m\\\lambda
\end{pmatrix},u'=\begin{pmatrix}
x_1'\\ \vdots\\x_m'\\\lambda
\end{pmatrix},u^*=\begin{pmatrix}
x_1^*\\ \vdots\\x_m^*\\\lambda^*
\end{pmatrix},F(u)=\begin{pmatrix}
-A_1^T\lambda\\\vdots\\-A_m^T\lambda\\Ax-b
\end{pmatrix},\\&v=\begin{pmatrix}
A_2x_2\\ \vdots\\A_mx_m\\\lambda
\end{pmatrix},v'=\begin{pmatrix}
A_2x_2'\\ \vdots\\A_mx_m'\\\lambda
\end{pmatrix},v^*=\begin{pmatrix}
A_2x_2^*\\ \vdots\\A_mx_m^*\\\lambda^*
\end{pmatrix},
x'=\begin{pmatrix}
x_1'\\ \vdots\\x_m'
\end{pmatrix}\in\Omega.
\end{aligned}
\end{equation}
We will always  use the above definitions of $u,~v$  and  $x',~u',~v'$  for \eqref{P1}-\eqref{P3}, unless explicitly stated otherwise.
As shown in \cite{he2018my,he2014strictly,he2016convergence,2012On}, the saddle point $(x^*,\lambda^*)$ can be alternatively characterized as a solution point of the
following variational inequality (VI):
\begin{equation}
\label{V1}
f(x)-f(x^*)+(u-u^*)^TF(u^*)\geq0,~\forall u\in \mathbb{R}^{n+l}.
\end{equation}

\subsection{Generalized prediction-correction framework with ergodic convergence rates}\label{S2.2}
To improve the ergodic convergence rate for solving \eqref{V1}, our previous work \cite{zhang2022faster} established a generalized framework with dynamically updated matrices $Q^k$ and $M^k$.
\begin{framed}
	\noindent{\bf[Generalized prediction step.]} With a given $v^k$, find $\widetilde{u}^k$ such that
	\begin{equation}\label{V7} f(x)-f(\widetilde{x}^k)+(u-\widetilde{u}^k)^TF(\widetilde{u}^k)\geq(v-\widetilde{v}^k)^TQ^k(v^k-\widetilde{v}^k)+\frac{\sigma}{2}\|z^k-z\|^2_R,
	~\forall u,
	\end{equation}
	where $\sigma\geq0$, $R\succeq 0$ and $(Q^k)^T +Q^k\succeq0$ (noting that $Q^k$ is not necessarily symmetric).
	
	\noindent {\bf [Generalized correction step.]} Update $v^{k+1}$ by
	\begin{equation}\label{V8}
	v^{k+1}=v^k-M^k(v^k-\widetilde{v}^k).
	\end{equation}
\end{framed}

If $\sigma=0$, $Q^k$ and  $M^k$ are fixed, it is exactly the framework presented by He and Yuan  \cite{2012On}.  The new variable $z$ will be  set as $x$, $x_i$,  $\nabla f(x_i)$ and so on.
Convergence is built under the following conditions.

\begin{framed}
	\noindent{\bf[Generalized convergence condition.]}
	For the matrices $Q^k$ and $M^k$ used in \eqref{V7} and \eqref{V8}, respectively, there exists a matrix $H^k$ such that
	\begin{equation}
	Q^k=H^kM^k,\label{V9}\tag{CC1}
	\end{equation}
	\begin{equation}
	\label{V10}\tag{CC2}G^k:=(Q^k)^T +Q^k-(M^k)^TH^kM^k.
	\end{equation}
\end{framed}

\begin{framed}
	\noindent {\bf[Additional convergence condition.]}
	For $r^k>0$ and $H_0^k\succeq0$, it holds that
	\begin{equation}\label{V12}\tag{CC3}
	\begin{aligned}
	&r^k\left(\|v^{k+1}-v' \|^2_{H^k}+\sigma\|z^k-z'\|^2_R-\|v^k-v' \|^2_{H^k}+\|v^k-\widetilde{v}^k\|^2_{G^k}\right)\\& \geq\|v^{k+1}-v'\|^2_{H_0^{k+1}}-\|v^k-v' \|^2_{H_0^k}+\varTheta^{k+1}-\varTheta^k,~\varTheta^k\geq0.
	\end{aligned}
	\end{equation}
\end{framed}

\begin{theorem}[\cite{zhang2022faster}] \label{main}
	Under the convergence conditions \eqref{V9}-\eqref{V12}, for the generalized prediction-correction framework \eqref{V7}-\eqref{V8}, we have,
	$$
	f(\widetilde{X}^K)-f(x')-\lambda^T(A\widetilde{X}^K-b) \leq O\left(1 \Big/ \sum_{k=0}^{K}r^k \right),
	$$
	where $\widetilde{X}^K=(\sum_{k=0}^{K}r^k\widetilde{x}^k)/(\sum_{k=0}^{K}r^k)$. In particular, setting $r^k=O(k)$ achieves  $O(1/K^2)$ convergence rate.
\end{theorem}


\section{Prediction-correction framework with $O(1/k)$ non-ergodic convergence rates} \label{subsubs5.1.1}
We first consider solving ${\rm \eqref{P1}}$.
For convenience, we define the differentiable part of the augmented Lagrangian function of \eqref{P1}:
\begin{equation}\label{G80}
\varphi^k(x,\lambda):=-\lambda^T(Ax-b)+\frac{\beta^k}{2}\|Ax-b\|^2,~\beta^k>0.
\end{equation}
It is not difficult to verify that $\varphi^k(x,\lambda)$ is $\beta^k$-gradient Lipschitz continuous with $\|\cdot\|_D$ in $x$, where $D=A^TA~{\rm or} ~\|A\|^2I_n .$
We introduce Nesterov's accelerated gradient method to  minimize the $x$-subproblem of the augmented Lagrangian function of  ${\rm \eqref{P1}}$:
\begin{eqnarray}\label{G2}
\begin{cases}
\hat{x}^k=x^k+\frac{\tau^k(1-\tau^{k-1})}{\tau^{k-1}}(x^k-x^{k-1}),~\tau^k>0,\\
x^{k+1}\in\arg\min\limits_x \{f(x)+x^T\nabla_x\varphi^k(\hat{x}^k,\lambda^k)+\frac{\beta^k}{2}\|x-\hat{x}^k\|_D^2
\}.
\end{cases}
\end{eqnarray}
The selection of $\lambda^k$ is sensitive to guarantee  convergence.
It motivates us to consider a prediction-correction framework. For convenience, we
set
\begin{equation}
1/\tau^{k-1}=(1-\tau^k)/\tau^k,~ \tau^{-1}\in(0,1) \label{G13},\tag{C1}
\end{equation}
from which one can observe that $\tau^k=O(1/k)$ for $k\rightarrow\infty$.
We define\begin{eqnarray}
\bar{x}^{k+1}:=x^{k+1}/\tau^k-(1-\tau^k){x}^k/\tau^k.\label{G12}
\end{eqnarray}

By the optimality condition of the $x$-subproblem in \eqref{G2}, we have
\begin{equation}\label{G4}
\begin{aligned}
&f(x)-f(x^{k+1})+(x-x^{k+1})^T[-A^T\lambda^k\\+&\beta^kA^T(A\hat{x}^{k}-b)+\beta^kD(x^{k+1}-\hat{x}^k)]\geq0,~\forall x.
\end{aligned}
\end{equation}
Multiplying both sides of \eqref{G4} by $(1-\tau^k)/\tau^k$ with $x=x^k$  and then adding it to \eqref{G4} yields that
\begin{equation}\label{G6}
\begin{aligned}
&\frac{1}{\tau^k}[f(x)-f(x^{k+1})]-\frac{1}{\tau^{k-1}}[f(x)-f(x^{k})]
+(x-\bar{x}^{k+1})^T[-A^T\lambda^k\\&+\beta^kA^T(A\hat{x}^{k}-b)+\beta^kD(x^{k+1}-\hat{x}^k)   ]\geq0,~\forall x.
\end{aligned}
\end{equation}
According to  the  definitions of $\bar{x}^{k+1}$ and $\hat{x}^k$,  \eqref{G6} is equivalent to
\begin{equation}\label{G7}
\begin{aligned}
&\frac{1}{\tau^k}[f(x)-f(x^{k+1})]-\frac{1}{\tau^{k-1}}[f(x)-f(x^{k})]
+(x-\bar{x}^{k+1})^T[-A^T\lambda^k\\&+\tau^k\beta^kA^T(A\bar{x}^{k+1}-b)+\tau^k\beta^k(D-A^TA)(\bar{x}^{k+1}-\bar{x}^k)\\&+(1-\tau^k) \beta^kA^T(A{x}^{k}-b) ]\geq0,~\forall x.
\end{aligned}
\end{equation}
To introduce the prediction-correction framework, the definitions of $u$ and $F(u)$ defined  in \eqref{V2} are adopted. We define $v^k:=(\bar{x}^k,{\lambda}^k)$,  the artificial vectors $\breve{x}^k$ and  $\widetilde{v}^k:=(\widetilde{x}^k,\widetilde{\lambda}^k)$ as
\begin{equation}\label{G5}
\breve{x}^k:=x^{k+1},~
\widetilde{x}^k:=\bar{x}^{k+1},~{\rm and}~\widetilde{\lambda}^k:=\lambda^k-\tau^k\beta^k(A\widetilde{x}^k-b).\end{equation}
Then, by the definition of $\widetilde{\lambda}^k$ in \eqref{G5}, it holds that
\begin{equation}\label{G57}
(\lambda-\widetilde{\lambda}^k)^T[
(A\widetilde{x}^k-b)-\frac{1}{\tau^k\beta^k}( \lambda^k- \widetilde{\lambda}^k)	
]\geq0,~\forall \lambda.
\end{equation}
Combining \eqref{G7} and \eqref{G57}, we have
\begin{equation}\label{G15}
\begin{aligned}
&\frac{1}{\tau^k}[f(x)-f(\breve{x}^{k})]-\frac{1}{\tau^{k-1}}[f(x)-f(\breve{x}^{k-1})]+(u-\widetilde{u}^k)^TF(\widetilde{u}^k)\\&+(1-\tau^k)\beta^k(A(x-\widetilde{x}^k))^T(A\breve{x}^{k-1}-b)\geq(v-\widetilde{v}^k)Q^k(v^k-\widetilde{v}^k),~\forall u.
\end{aligned}
\end{equation}
where
\begin{equation}
\label{G19}
Q^k=\begin{pmatrix}
\tau^k\beta^k(D-A^TA)&0\\
0&\frac{1}{\tau^k\beta^k}I_l
\end{pmatrix}.
\end{equation}
Let us define
\begin{equation}
\label{G20}
\begin{aligned}
&M^k=\begin{pmatrix}
I_n&0\\
0&\gamma I_l
\end{pmatrix},~ H^k=\begin{pmatrix}
\tau^k\beta^k( D-A^TA)&0\\
0&\frac{1}{\gamma\tau^k\beta^k}I_l
\end{pmatrix}, \\& G^k= \begin{pmatrix}
\tau^k\beta^k(D-A^TA)&0\\
0&\frac{2-\gamma}{\tau^k\beta^k}I_l
\end{pmatrix},\\
&\lambda^{k+1}=\lambda^k-\gamma\tau^k\beta^k(A\bar{x}^{k+1}-b)
=\lambda^k-\gamma(\lambda^k-\widetilde{\lambda}^k),~\gamma>0.
\end{aligned}
\end{equation}
Then $H^k$ and $G^k$ satisfy the convergence conditions  \eqref{V9}-\eqref{V10} and
\begin{equation}\label{G21}
v^{k+1}=v^k-M^k(v^k-\widetilde{v}^k).
\end{equation}
Consequently, we obtain
\begin{equation}\label{G8}
\begin{aligned}
&\frac{1}{\tau^k}[f(x)-f(\breve{x}^{k})]-\frac{1}{\tau^{k-1}}[f(x)-f(\breve{x}^{k-1})]+(u-\widetilde{u}^k)^TF(\widetilde{u}^k)\\&+(1-\tau^k)\beta^k(A(x-\widetilde{x}^k))^T(A\breve{x}^{k-1}-b)\\\geq&
(v-\widetilde{v}^k)Q^k(v^k-\widetilde{v}^k)
\\=&\frac{1}{2}\Big(\|v^{k+1}-v\|^2_{H^k}-\|v^k-v\|^2_{H^k}+\|v^k-\widetilde{v}^k\|^2_{G^k}
\Big),~\forall u.
\end{aligned}
\end{equation}
Let $\tau^k\beta^k=\beta>0$. Then the convergence condition  \eqref{V12}
holds with $r^k=1$, $\sigma=0$, $H_0^k=H^k=\begin{pmatrix}
\beta(D-A^TA)&0\\
0&\frac{1}{\gamma\beta}I_l
\end{pmatrix}$ and $\varTheta^{k+1}-\varTheta^{k}=\|v^k-\widetilde{v}^k\|^2_{G^k}$.  For convenience, we define \begin{equation}\label{G35}
S^{k+1}:=f(x')-f(\breve{x}^{k})+\lambda^T(A\breve{x}^{k}-b),~x'\in\Omega.
\end{equation}
Since \begin{eqnarray*}	(u'-\widetilde{u}^k)^TF(\widetilde{u}^k)=\lambda^T(A\bar{x}^{k+1}-b)&=& \frac{1}{\tau^k}\lambda^T(A{x}^{k+1}-b)-\frac{1}{\tau^{k-1}}\lambda^T(A{x}^{k}-b)\\&=&\frac{1}{\tau^k}\lambda^T(A\breve{x}^{k}-b)-\frac{1}{\tau^{k-1}}\lambda^T(A\breve{x}^{k-1}-b),
\end{eqnarray*}
substituting $x=x'$, $u=u'$ and $v=v'$ into
\eqref{G8} yields that
\begin{equation}\label{G9}
\begin{aligned}
&\frac{1}{\tau^k}S^{k+1}-\frac{1}{\tau^{k-1}}S^k+(1-\tau^k)\beta^k(b-A\widetilde{x}^k)^T(A\breve{x}^{k-1}-b)
\\\geq&\frac{1}{2}\Big(\|v^{k+1}-v'\|^2_{H^k}-\|v^k-v'\|^2_{H^k}+\|v^k-\widetilde{v}^k\|^2_{G^k}\Big)
\\\geq&\frac{1}{2}\Big(\|v^{k+1}-v'\|^2_{H^{k+1}_0}-\|v^k-v'\|^2_{H^k_0}+\|v^k-\widetilde{v}^k\|^2_{G^k}
\Big).
\end{aligned}
\end{equation}
Note that we can verify that
\begin{eqnarray}
&&\|v^k-\widetilde{v}^k\|^2_{G^k}\geq(2-\gamma)\tau^k\beta^k\|A\widetilde{x}^{k}-b\|^2,\label{G10}\\
&&\|A{x}^{k+1}-b\|^2=(\tau^k)^2\|A\bar{x}^{k+1}-b\|^2+(1-\tau^k)^2\|A{x}^{k}-b\|^2\nonumber\\
&&+2\tau^k(1-\tau^k)(A\bar{x}^{k+1}-b)^T(Ax^k-b).\label{G11}
\end{eqnarray}
The equalities \eqref{G5}, \eqref{G11} and \eqref{G13} imply that
\begin{equation}\label{G14}
\begin{aligned}
\frac{1}{(\tau^k)^2}	\|A\breve{x}^{k}-b\|^2=&\|A\widetilde{x}^{k}-b\|^2+\frac{1}{(\tau^{k-1})^2}\|A\breve{x}^{k-1}-b\|^2\\
&+2\frac{1-\tau^k}{\tau^{k}}(A\widetilde{x}^{k}-b)^T(A\breve{x}^{k-1}-b).
\end{aligned}
\end{equation}
If $\gamma=1$, it follows from \eqref{G9}, \eqref{G10} and \eqref{G14} that
\begin{equation}
\begin{aligned}
&\frac{1}{\tau^k}[S^{k+1}-\frac{\beta}{2\tau^k}\|A\breve{x}^{k}-b\|^2]-
\frac{1}{\tau^{k-1}}[S^{k}-\frac{\beta}{2\tau^{k-1}}\|A\breve{x}^{k-1}-b\|^2]\\
\geq&
\frac{1}{2}\Big(\|v^{k+1}-v'\|^2_{H^{k+1}_0}-\|v^k-v'\|^2_{H^k_0}
\Big).
\end{aligned}	
\end{equation}
That is, the sequence
\[ 	\left\{\frac{1}{\tau^{k-1}}[S^{k}-\frac{\beta}{2\tau^{k-1}}\|A\breve{x}^{k-1}-b\|^2]-\frac{1}{2}\|v^k-v'\|^2_{H_0^k} \right\}
\]
is monotonically non-decreasing. Consequently,
we can establish $O(1/k)$ non-ergodic convergence rate  of
$f(\breve{x}^{k})-f(x')-\lambda^T(A\breve{x}^{k}-b).$

Now we can summarize the above analysis as the following prediction-correction framework.
\begin{framed}
	\noindent{\bf[Prediction step.]} With given $\breve{x}^{k-1}$ and $v^k$, find $\breve{x}^k$ and $\widetilde{u}^k$  such that
	\begin{equation}\label{G27}\tag{PS1}
	\begin{aligned}
	&\frac{1}{\tau^k}[f(x)-f(\breve{x}^{k})]-\frac{1}{\tau^{k-1}}
	[f(x)-f(\breve{x}^{k-1})]+(u-\widetilde{u}^k)^TF(\widetilde{u}^k)\\&+c^k(A(x-\widetilde{x}^k))^T(A\breve{x}^{k-1}-b)\geq(v-\widetilde{v}^k)^TQ^k(v^k-\widetilde{v}^k),~\forall u,
	\end{aligned}
	\end{equation}
	where
	$c^k\geq0$,  $\widetilde{x}^k=\frac{1}{\tau^k}\breve{x}^{k}-\frac{1-\tau^k}{\tau^k}\breve{x}^{k-1}$, and $\tau^k$  satisfies \eqref{G13}.
	
	\noindent{\bf [Correction step.]} Update $v^{k+1}$ by
	\begin{equation}\label{G28}\tag{CS1}
	v^{k+1}=v^k-M^k(v^k-\widetilde{v}^k).
	\end{equation}
\end{framed}

The convergence is summarized in the following lemma without additional proof.

\begin{lemma}\label{L5.1}
	Let $\beta^k=\beta/\tau^k$ ($\beta>0$) and $c^k=c(1-\tau^k)\beta^k$ ($c\geq 0$). If $\{v^{k+1}\}$ generated by the prediction-correction framework \eqref{G27}-\eqref{G28} satisfies the convergence conditions $\eqref{V9}$-$\eqref{V10}$ and $\eqref{V12}$ with $r^k=1$, $\sigma=0$, and $\varTheta^{k+1}-\varTheta^{k}\geq c\tau^k\beta^k\|A\widetilde{x}^k-b\|^2$,
	then it holds that
	$$\begin{aligned}
	f(\breve{x}^{k})-f(x')-\lambda^T(A\breve{x}^{k}-b)\leq O(1/k).
	\end{aligned}
	$$
\end{lemma}

Throughout the following of  section \ref{subsubs5.1.1},  we adopt the definitions of $Q^k$ in \eqref{G19};  $M^k$, $H^k$, $G^k$ in \eqref{G20};
$u$, $F(u)$  in \eqref{V2} and $\tau^k$ in \eqref{G13}.
Next, we show how to construct algorithms satisfying  the general case  $M^k$, $H^k$, and $G^k$ with $\gamma\in(0,2]$. It suffices to verify the conditions presented in Lemma \ref{L5.1}.

\subsection{Correcting multiplier twice}

We consider updating $\{x^{k+1}\}$   in the following:
\begin{eqnarray}\label{G120}
\begin{cases}
\hat{\lambda}^k=\lambda^k-(1-\tau^k)\beta^k(A{x}^{k}-b),\\
\hat{x}^k=x^k+\frac{\tau^k(1-\tau^{k-1})}{\tau^{k-1}}(x^k-x^{k-1}),\\
x^{k+1}\in\arg\min\limits_x \{f(x)+x^T\nabla_x\varphi^k(\hat{x}^k,\hat{\lambda}^k)+\frac{\beta^k}{2}\|x-\hat{x}^k\|_D^2\}.
\end{cases}
\end{eqnarray}
Define
$u^k=v^k:=(\bar{x}^k,\bar{\lambda}^k)$ with $\{\bar{x}^{k}\}$   in \eqref{G12},
the artificial vectors $\breve{x}^k$ and  $\widetilde{v}^k:=(\widetilde{x}^k,\widetilde{\lambda}^k)$ as
\begin{eqnarray}
\label{G22}
\bm{{\rm(the ~first ~correction)}}~~
\bar{\lambda}^k:=\lambda^k-\gamma(1-\tau^k)\beta^k(A\breve{x}^{k-1}-b),\\
\label{G16}
\breve{x}^k:=x^{k+1},~
\widetilde{x}^k:=\bar{x}^{k+1},~{\rm and}~\widetilde{\lambda}^k:=\bar{\lambda}^k-\tau^k\beta^k(A\widetilde{x}^k-b).\end{eqnarray}

According to the  optimality condition of the $x$-subproblem in \eqref{G120}, we obtain
\begin{equation}\label{G121}
\begin{aligned}
&f(x)-f({x}^{k+1})+(x-{x}^{k+1})^T[-A^T\hat{\lambda}^k \\&+\beta^kA^T(A\hat{x}^{k}-b)+	\beta^kD({x}^{k+1}-\hat{x}^k)
]\geq0,~\forall x.\\\Leftrightarrow
&	f(x)-f(\breve{x}^{k})+(x-\breve{x}^{k})^T[
-A^T{\lambda}^k\\&+\gamma(1-\tau^k)\beta^k   A^T(A\breve{x}^{k-1}-b)+\tau^k\beta^k(D-A^TA)(\widetilde{x}^{k}-\bar{x}^k)\\&+\tau^k\beta^kA^T(A\widetilde{x}^{k}-b)
+(2-\gamma)(1-\tau^k) \beta^kA^T(A\breve{x}^{k-1}-b)
]\geq0,~\forall x.
\end{aligned}
\end{equation}
Then multiplying both sides of \eqref{G121} by $(1-\tau^k)/\tau^k$ with $x=\breve{x}^{k-1}$  and then adding it to \eqref{G121} yields that
\begin{equation}\label{G17}
\begin{aligned}
&\frac{1}{\tau^k}[f(x)-f(\breve{x}^{k})]-\frac{1}{\tau^{k-1}}[f(x)-f(\breve{x}^{k-1})]
+(x-\widetilde{x}^{k})^T[-A^T\widetilde{\lambda}^k \\&	+\tau^k\beta^k(D-A^TA)(\widetilde{x}^{k}-\bar{x}^k)
+(2-\gamma)(1-\tau^k) \beta^kA^T(A\breve{x}^{k-1}-b) ]\geq0,~\forall x.
\end{aligned}
\end{equation}
It follows from the definition of  $\widetilde{\lambda}^k$ \eqref{G16} that
\begin{equation}\label{G84}
(\lambda-\widetilde{\lambda}^k)^T[
(A\widetilde{x}^k-b)-\frac{1}{\tau^k\beta^k}(\bar{ \lambda}^k- \widetilde{\lambda}^k)	
]\geq0,~\forall \lambda.
\end{equation}
Combining \eqref{G17} and \eqref{G84} yields that
\begin{equation}\label{G18}
\begin{aligned}
&\frac{1}{\tau^k}[f(x)-f(\breve{x}^{k})]-\frac{1}{\tau^{k-1}}[f(x)-f(\breve{x}^{k-1})]+(u-\widetilde{u}^k)^TF(\widetilde{u}^k)\\+&(2-\gamma)(1-\tau^k)\beta^k(A(x-\widetilde{x}^k))^T(A\breve{x}^{k-1}-b)\geq(v-\widetilde{v}^k)Q^k(v^k-\widetilde{v}^k),~\forall u,
\end{aligned}
\end{equation}
where $Q^k$ is defined in \eqref{G19}. Then  the prediction step \eqref{G27} holds with $c^k=(2-\gamma)(1-\tau^k)\beta^k$. If we set the dual update rule as
\begin{equation}\label{G23}
\bm{{\rm(the ~second ~correction)}}~~
\bar{\lambda}^{k+1}=\bar{\lambda}^k-\gamma\tau^k\beta^k(A\bar{x}^{k+1}-b),~\gamma\in(0,2],
\end{equation}
then $v^{k+1}$ satisfies the correction step \eqref{G28} with $M^k$ defined in \eqref{G20}.

We can verify the conditions \eqref{V9}-\eqref{V10}. Let $\tau^k\beta^k=\beta>0$. Then, according to the definition of $G^k$,  $\|v^k-\widetilde{v}^k\|^2_{G^k}$ in \eqref{G10} holds.
We can verify the condition \eqref{V12}.

In this situation, according to \eqref{G22} and \eqref{G23},
we can verify that
\begin{eqnarray*}
	\bar{\lambda}^{k+1}&\overset{\eqref{G22}}{=}&\lambda^{k+1}-\gamma(1-\tau^{k+1})\beta^{k+1}(A\breve{x}^{k}-b),\\\bar{\lambda}^{k+1}&\overset{\eqref{G22},\eqref{G23}}{=}&\lambda^k-\gamma(1-\tau^k)\beta^k(A\breve{x}^{k-1}-b)-\gamma\tau^k\beta^k(A\bar{x}^{k+1}-b).
\end{eqnarray*}
Therefore, $\lambda^{k+1}$ is updated from $\lambda^k$ by
\begin{eqnarray}
\lambda^{k+1}&=&\lambda^k-\gamma[ (1-\tau^{k})\beta^k(Ax^{k}-b)-(1-\tau^{k+1})\beta^{k+1}(Ax^{k+1}-b) ]\nonumber\\&&-\gamma\tau^k\beta^k(A\bar{x}^{k+1}-b)\nonumber\\&\overset{\tau^k\beta^k=\beta,\eqref{G13}}{=}&\lambda^k.\nonumber
\end{eqnarray}
We obtain a penalty method if $\lambda^0=0$. More discussions on the equivalence  between $correcting ~v^k~ twice$ and the penalty method for solving \eqref{P2} and \eqref{P3} are presented  in Theorems \ref{T4.5} and  \ref{T4.8}.

\subsection{Correcting multiplier once}
We consider updating $\{x^{k+1}\}$  in the following:
\begin{eqnarray}\label{G110}
\begin{cases}
\hat{\lambda}^k=\lambda^k-(1-\gamma)(1-\tau^k)\beta^k(A{x}^{k}-b),\\
\hat{x}^k=x^k+\frac{\tau^k(1-\tau^{k-1})}{\tau^{k-1}}(x^k-x^{k-1}),\\
x^{k+1}\in\arg\min\limits_x \{f(x)+x^T\nabla_x\varphi^k(\hat{x}^k,\hat{\lambda}^k)+\frac{\beta^k}{2}\|x-\hat{x}^k\|_D^2\}.
\end{cases}
\end{eqnarray}
We define $v^k:=(\bar{x}^k,{\lambda}^k)$ with $\bar{x}^{k}$  in  \eqref{G12},  the artificial vectors $\breve{x}^k$ and  $\widetilde{v}^k:=(\widetilde{x}^k,\widetilde{\lambda}^k)$ in \eqref{G5}.

According to the  optimality condition of the $x$-subproblem in \eqref{G110}, we obtain
\begin{equation}\label{G112}
\begin{aligned}
&f(x)-f({x}^{k+1})+(x-{x}^{k+1})^T[-A^T\hat{\lambda}^k \\&+\beta^kA^T(A\hat{x}^{k}-b)+	\beta^kD({x}^{k+1}-\hat{x}^k)
]\geq0,~\forall x.\\\Leftrightarrow
&	f(x)-f(\breve{x}^{k})+(x-\breve{x}^{k})^T[
-A^T{\lambda}^k+\tau^k\beta^k(D-A^TA)(\widetilde{x}^{k}-\bar{x}^k)\\&+\tau^k\beta^kA^T(A\widetilde{x}^{k}-b)
+(2-\gamma)(1-\tau^k) \beta^kA^T(A\breve{x}^{k-1}-b) ]\geq0,~\forall x.
\end{aligned}
\end{equation}
Multiplying both sides of \eqref{G112} by $(1-\tau^k)/\tau^k$ with $x=\breve{x}^{k-1}$  and then adding it to \eqref{G112} yields that
\begin{equation}\label{G111}
\begin{aligned}
&\frac{1}{\tau^k}[f(x)-f(\breve{x}^{k})]-\frac{1}{\tau^{k-1}}[f(x)-f(\breve{x}^{k-1})]
+(x-\widetilde{x}^{k})^T[-A^T\widetilde{\lambda}^k\\&
+\tau^k\beta^k(D-A^TA)(\widetilde{x}^{k}-\bar{x}^k)
+(2-\gamma)(1-\tau^k) \beta^kA^T(A\breve{x}^{k-1}-b) ]\geq0,~\forall x.
\end{aligned}
\end{equation}
By  the definition of $\widetilde{\lambda}^k$ in \eqref{G5} (or \eqref{G57}), we can show that the prediction step \eqref{G27} holds with
$c^k=(2-\gamma)(1-\tau^k)\beta^k$.  If  $\lambda^{k+1}$ satisfies \eqref{G20} with $\gamma\in(0,2]$, the  correction step holds.
The conditions \eqref{V9}-\eqref{V10} and \eqref{V12} can also be verified by using $\tau^k\beta^k=\beta>0$.

Different from  $correcting ~multiplier~ twice$ with fixed $\lambda^k$, $correcting ~multiplier$ $once$  updates $\lambda^k$.

\section{Prediction-correction framework with $O(1/k^2)$ non-ergodic convergence rates} \label{s4.2}

This section aims at establishing $O(1/k^2)$ non-ergodic convergence rates under the condition:
\begin{equation}\label{G29}\tag{C2}
1/(\tau^{k-1})^2=(1-\tau^k)/(\tau^k)^2,~ \tau^{-1}\in(0,1).
\end{equation}
Clearly, it implies from \eqref{G29}  by induction that $1/\tau^k\geq (k+1)/2$.

Motivated by the framework with ergodic convergence rates presented in subsection \ref{S2.2} and the framework with $O(1/k)$ non-ergodic convergence presented in section \ref{subsubs5.1.1}, we present the following prediction-correction framework:
\begin{framed}
	\noindent{\bf[Prediction step.]}  With given $\breve{x}^{k-1}$ and $v^k$, find $\breve{x}^k$ and $\widetilde{u}^k$  such that
	\begin{equation}\label{G54}\tag{PS2}
	\begin{aligned}
	&\frac{1}{(\tau^k)^2}[f(x)-f(\breve{x}^{k})]-\frac{1}{(\tau^{k-1})^2}[f(x)-f(\breve{x}^{k-1})]
	\\
	&+\frac{1}{\tau^k}(u-\widetilde{u}^k)^TF(\widetilde{u}^k)+ c^k(A(x-\widetilde{x}^k))^T(A\breve{x}^{k-1}-b)\\
	&\geq\frac{1}{\tau^k}\Big[(v-\widetilde{v}^k)^TQ^k(v^k-\widetilde{v}^k)
	+\frac{\sigma}{2}\|z^{k}-z\|^2_R\Big],
	~\forall u,
	\end{aligned}
	\end{equation}
	where
	$c^k\geq0$, $\sigma\geq0$,
	$\widetilde{x}^k=\frac{1}{\tau^k}\breve{x}^{k}-\frac{1-\tau^k}{\tau^k}\breve{x}^{k-1}$,  and $\tau^k$ satisfies \eqref{G29}.
	
	\noindent {\bf [Correction step.]} Update $v^{k+1}$ by
	\begin{equation}\label{G55}\tag{CS2}
	v^{k+1}=v^k-M^k(v^k-\widetilde{v}^k).
	\end{equation}
\end{framed}

The following result
inherits from Lemma \ref{L5.1}.
\begin{lemma}\label{L5.2}
	Let $\beta^k=\beta/(\tau^k)^2$ ($\beta>0$) and $c^k=c(1-\tau^k)\beta^k/{\tau^k}$  ($c\geq0$).	If $\{v^{k+1}\}$ generated by the prediction-correction framework \eqref{G54}-\eqref{G55} satisfies the convergence conditions $\eqref{V9}$-$\eqref{V10}$ and $\eqref{V12}$ with $r^k=1/\tau^k$ and $\varTheta^{k+1}-\varTheta^{k}\geq c\beta^k\|A\widetilde{x}^k-b\|^2$,
	then it holds that
	$$\begin{aligned}
	f(\breve{x}^{k})-f(x')-\lambda^T(A\breve{x}^{k}-b)\leq O(1/k^2).
	\end{aligned}
	$$
\end{lemma}
\begin{proof}
	According to the conditions given in Lemma \ref{L5.2},  by substituting $x=x',~z=z',~u=u'$ and $v=v'$ into \eqref{G54}, we have
	\begin{equation}\label{G34}
	\begin{aligned}
	&\frac{1}{(\tau^k)^2}[f(x')-f(\breve{x}^{k})]-\frac{1}{(\tau^{k-1})^2}[f(x')-f(\breve{x}^{k-1})]+\frac{1}{\tau^k}(u'-\widetilde{u}^k)^TF(\widetilde{u}^k)\\&+c^k(b-A\widetilde{x}^k)^T(A\breve{x}^{k-1}-b)\\\geq&\frac{1}{\tau^k}\Big((v'-\widetilde{v}^k)^TQ^k(v^k-\widetilde{v}^k)
	+\frac{\sigma}{2}\|z^{k}-z'\|^2_R\Big)\\
	\geq&\frac{1}{2\tau^k}\Big(\|v^{k+1}-v' \|^2_{H^k} +\sigma\|{z}^{k}-z'\|^2_R -\|v^k-v' \|^2_{H^k}+\|v^k-\widetilde{v}^k\|^2_{G^k}\Big)
	\\\geq&
	\frac{1}{2}\Big(\|v^{k+1}-v'\|^2_{H^{k+1}_0}-\|v^k-v'\|^2_{H^k_0}+\frac{c\beta}{(\tau^k)^2}\|A\widetilde{x}^k-b\|^2\Big).
	\end{aligned}
	\end{equation}
	According to \eqref{G29},  we obtain
	\begin{eqnarray}
	&&\begin{aligned}	\frac{1}{\tau^k}(u'-\widetilde{u}^k)^TF(\widetilde{u}^k)&=\frac{1}{\tau^k}\lambda^T(A\widetilde{x}^{k}-b)\\
	&= \frac{1}{(\tau^k)^2}\lambda^T(A\breve{x}^{k}-b)-\frac{1}{(\tau^{k-1})^2}\lambda^T(A\breve{x}^{k-1}-b),
	\end{aligned}\label{G33}\\
	&&\begin{aligned}
	\frac{1}{(\tau^k)^4}	\|A\breve{x}^{k}-b\|^2&=\frac{1}{(\tau^k)^2}\|A\widetilde{x}^{k}-b\|^2+\frac{1}{(\tau^{k-1})^4}\|A\breve{x}^{k-1}-b\|^2\\
	&~~+2\frac{1-\tau^k}{(\tau^{k})^3}(A\widetilde{x}^{k}-b)^T(A\breve{x}^{k-1}-b).
	\end{aligned}\label{G32}
	\end{eqnarray} 		
	Combining \eqref{G34}, \eqref{G33} and \eqref{G32} yields that
	\begin{equation}\nonumber
	\begin{aligned}
	&\frac{1}{(\tau^k)^2}[S^{k+1}-\frac{c\beta }{2(\tau^k)^2}\|A\breve{x}^{k}-b\|^2]-
	\frac{1}{(\tau^{k-1})^2}[S^{k}-\frac{c\beta}{2(\tau^{k-1})^2}\|A\breve{x}^{k-1}-b\|^2]\\
	\geq&
	\frac{1}{2}\Big(\|v^{k+1}-v'\|^2_{H^{k+1}_0}-\|v^k-v'\|^2_{H^k_0}
	\Big),
	\end{aligned}
	\end{equation} 	
	where $S^k$ is defined in \eqref{G35}. We complete the proof.
\end{proof}

In Lemmas \ref{L5.1} and \ref{L5.2}, our settings satisfy $\beta^k\rightarrow+\infty$ as $k\rightarrow+\infty$. In the following result, we consider the setting $\beta^k\equiv\beta>0$. It is used to establish $O(1/k^2)$ non-ergodic convergence rate for  solving \eqref{P3}.
\begin{lemma}\label{L5.3}
	Let $\beta^k\equiv\beta>0$ and  $c^k=c(1-\tau^k)\beta^k/{\tau^k}$ ($c\geq 0$). If the sequence $\{v^{k+1}\}$ generated by the prediction-correction framework \eqref{G54}-\eqref{G55} satisfies the convergence conditions $\eqref{V9}$-$\eqref{V10}$ and $\eqref{V12}$ with $r^k=1/\tau^k$ and $\varTheta^{k+1}-\varTheta^{k}\geq c\beta^k(\|A\widetilde{x}^k-b\|^2
	-  \frac{\tau^k}{(\tau^{k-1})^2}\|A\breve{x}^{k-1}-b\|^2
	)$,
	then it holds that
	$$\begin{aligned}
	f(\breve{x}^{k})-f(x')-\lambda^T(A\breve{x}^{k}-b)\leq O(1/k^2).
	\end{aligned}
	$$
\end{lemma}

\begin{proof}
	According to the definition of $\widetilde{x}^k$, we obtain
	\begin{eqnarray}
	&&\frac{1}{(\tau^k)^2}	\|A\breve{x}^{k}-b\|^2\nonumber\\&=&\|A\widetilde{x}^{k}-b\|^2+\frac{(1-\tau^k)^2}{(\tau^k)^2}\|A\breve{x}^{k-1}-b\|^2 +2\frac{1-\tau^k}{\tau^{k}}(A\widetilde{x}^{k}-b)^T(A\breve{x}^{k-1}-b)\nonumber\\&\overset{\eqref{G29}}{=}&\|A\widetilde{x}^{k}-b\|^2+\frac{1}{(\tau^{k-1})^2}\|A\breve{x}^{k-1}-b\|^2  -  \frac{\tau^k}{(\tau^{k-1})^2}\|A\breve{x}^{k-1}-b\|^2\nonumber \\&&+2\frac{1-\tau^k}{\tau^{k}}(A\widetilde{x}^{k}-b)^T(A\breve{x}^{k-1}-b). \label{G68}
	\end{eqnarray} 	
	Since the convergence condition  $\eqref{V12}$ holds  with $r^k=1/\tau^k$ and $\varTheta^{k+1}-\varTheta^{k}=c\beta^k(\|A\widetilde{x}^k-b\|^2
	-  \frac{\tau^k}{(\tau^{k-1})^2}\|A\breve{x}^{k-1}-b\|^2
	)$,  we have
	\begin{equation}\label{G69}
	\begin{aligned}
	&\frac{1}{(\tau^k)^2}[f(x')-f(\breve{x}^{k})]-\frac{1}{(\tau^{k-1})^2}[f(x')-f(\breve{x}^{k-1})]+\frac{1}{\tau^k}(u'-\widetilde{u}^k)^TF(\widetilde{u}^k)\\&+c^k(b-A\widetilde{x}^k)^T(A\breve{x}^{k-1}-b)\\\geq&\frac{1}{\tau^k}\Big((v'-\widetilde{v}^k)^TQ^k(v^k-\widetilde{v}^k)
	+\frac{\sigma}{2}\|z^{k}-z'\|^2_R\Big)\\\geq&
	\frac{1}{2}\Big(\|v^{k+1}-v'\|^2_{H^{k+1}_0}-\|v^k-v'\|^2_{H^k_0}+c\beta \|A\widetilde{x}^{k}-b\|^2  -\frac{c\beta \tau^k}{(\tau^{k-1})^2}\|A\breve{x}^{k-1}-b\|^2
	\Big).
	\end{aligned}
	\end{equation}
	Then combining \eqref{G68} and \eqref{G69} yields that
	\begin{equation}\nonumber
	\begin{aligned}
	&	\frac{1}{(\tau^k)^2}[S^{k+1}-\frac{c\beta }{2}\|A\breve{x}^{k}-b\|^2]-
	\frac{1}{(\tau^{k-1})^2}[S^{k}-\frac{c\beta }{2}\|A\breve{x}^{k-1}-b\|^2]\\
	\geq& \frac{1}{2}\Big(\|v^{k+1}-v'\|^2_{H^{k+1}_0}-\|v^k-v'\|^2_{H^k_0}
	\Big),
	\end{aligned}
	\end{equation}
	where $S^k$ is defined in \eqref{G35}. The proof is complete.
\end{proof}

Throughout the following of section \ref{s4.2},  we adopt the definitions of $Q^k$ in \eqref{G19};  $M^k$, $H^k$, $G^k$ in \eqref{G20};
$u$, $F(u)$  in \eqref{V2} and $\tau^k$ in \eqref{G29}. We assume that $f$ is $\sigma$-strongly convex with $\sigma\geq0$.

\subsection{Correcting  multiplier twice}\label{s4.2.1}
We consider an alternative way to
solve the $x$-subproblem of the augmented Lagrangian function of  ${\rm \eqref{P1}}$:
\begin{equation}\label{G85}
\begin{cases}
\hat{\lambda}^k=\lambda^k-(1-\tau^k)\beta^k(A{x}^{k}-b),\\
\hat{x}^k=x^k+\frac{\tau^k(1-\tau^{k-1})}{\tau^{k-1}}(x^k-x^{k-1}),\\
x^{k+1}\in\arg\min\limits_x\{ f(x)+x^T\nabla_x\varphi^k(\hat{x}^k,\hat{\lambda}^k)+\frac{\beta^k}{2}\|x-\hat{x}^k\|_D^2+\frac{\sigma(1-\tau^k)}{2\tau^k}\|x-x^k\|^2\}.
\end{cases}
\end{equation}
Below we verify that \eqref{G85} satisfies the prediction-correction framework \eqref{G54}-\eqref{G55} with properly dual updating.

Let
$u^k=v^k:=(\bar{x}^k,\bar{\lambda}^k)$ with
$\bar{x}^{k}$ in \eqref{G12} and $\bar{\lambda}^k$  in \eqref{G22}. Let $\widetilde{v}^k:=(\widetilde{x}^k,\widetilde{\lambda}^k)$  with $\breve{x}^k$, $\widetilde{x}^k$, and $\widetilde{\lambda}^k$ in \eqref{G16}.

According to the optimality condition of the $x$-subproblem in \eqref{G85}, we have
\begin{equation}\label{G86}
\begin{aligned}
&f(x)-f(x^{k+1})+(x-x^{k+1})^T[-A^T\hat{\lambda}^k+\beta^kA^T(A\hat{x}^k-b)\\&+\beta^kD(x^{k+1}-\hat{x}^k)+\frac{\sigma(1-\tau^k)}{\tau^k}(x^{k+1}-x^k)]\geq\frac{\sigma}{2}\|x^{k+1}-x\|^2,~\forall x.
\end{aligned}
\end{equation}
Multiplying both sides of \eqref{G86} by $(1-\tau^k)/(\tau^k)^2$ and $1/\tau^k$, respectively, setting $x=\breve{x}^{k-1}$ in the former inequality, and then adding them yields that
\begin{small}
	\begin{equation}\label{G87}
	\begin{aligned}
	&\frac{1}{(\tau^k)^2}[f(x)-f(\breve{x}^{k})]-\frac{1}{(\tau^{k-1})^2}[f(x)-f(\breve{x}^{k-1})]
	+\frac{1}{\tau^k}(x-\widetilde{x}^{k})^T\{-A^T\widetilde{\lambda}^k+\\&\tau^k\beta^k(D-A^TA)(\widetilde{x}^{k}-\bar{x}^k)+
	(2-\gamma)(1-\tau^k) \beta^kA^T(A\breve{x}^{k-1}-b)+\frac{\sigma(1-\tau^k)}{\tau^k}(\breve{x}^k-\breve{x}^{k-1})\}\\\geq&
	\frac{\sigma}{2}\Big(
	\frac{1-\tau^k}{(\tau^k)^2}\|\breve{x}^{k}-\breve{x}^{k-1}\|^2+\frac{1}{\tau^k}\|\breve{x}^{k}-x\|^2
	\Big),~\forall x.
	\end{aligned}
	\end{equation}
\end{small}
It is not difficult to verify the following two equalities:
\begin{align}
& \frac{1-\tau^k}{(\tau^k)^2}\|\breve{x}^{k}-\breve{x}^{k-1}\|^2+\frac{1}{\tau^k}\|\breve{x}^{k}-x\|^2\label{G89}=\|\widetilde{x}^{k}-x\|^2+\frac{1-\tau^k}{\tau^k}\|\breve{x}^{k-1}-x\|^2,~\\&
2(\widetilde{x}^k-x)^T\frac{1}{\tau^k}(\breve{x}^{k}-\breve{x}^{k-1})=\|\widetilde{x}^k-x\|^2+\frac{1}{(\tau^k)^2}\|\breve{x}^{k}-\breve{x}^{k-1}\|^2-\|\breve{x}^{k-1}-x\|^2\nonumber
.
\end{align}
Then, combining \eqref{G87} with $\widetilde{\lambda}^k$ in \eqref{G16} yields that
\begin{equation}\label{G88}
\begin{aligned}
&\frac{1}{(\tau^k)^2}[f(x)-f(\breve{x}^{k})]-\frac{1}{(\tau^{k-1})^2}[f(x)-f(\breve{x}^{k-1})]
+\frac{1}{\tau^k}(u-\widetilde{u}^k)^TF(\widetilde{u}^k)\\&+\frac{1}{\tau^k}(2-\gamma)(1-\tau^k) \beta^k(A(x-\widetilde{x}^k))^T(A\breve{x}^{k-1}-b)\\&\geq
\frac{1}{\tau^k}\Big((v-\widetilde{v}^k)^TQ^k(v^k-\widetilde{v}^k) +\frac{\sigma}{2}\|\widetilde{x}^k-x\|^2+
\sigma\frac{1-\tau^k}{2(\tau^k)^2}\|\breve{x}^{k}-\breve{x}^{k-1}\|^2\Big),~\forall u.
\end{aligned}
\end{equation}
By \eqref{G88},  the prediction step \eqref{G54} holds with $c^k=(2-\gamma)(1-\tau^k) \beta^k/\tau^k$, $R=I_n$, $z^k=\widetilde{x}^k$ and $z=x$. If $\bar{\lambda}^{k+1}$ satisfies \eqref{G23},  then $v^{k+1}$ satisfies the correction step  \eqref{G55} with $M^k$ defined in \eqref{G20}.

If $D=A^TA$,
according to the definition of $G^k$, $\|v^k-\widetilde{v}^k\|^2_{G^k}$ satisfies \eqref{G10}.
If $\beta^k=\beta/(\tau^k)^2$ for $\beta>0$, the convergence condition \eqref{V12} holds with $H^k_0=\begin{pmatrix}
0&0\\
0&\frac{1}{\gamma\beta}I_l
\end{pmatrix}$, $r^k=1/\tau^k$, $\sigma=0$ and $\varTheta^{k+1}-\varTheta^{k}=(2-\gamma)\beta^k\|A\widetilde{x}^k-b\|^2$. Then we can  verify  all the conditions presented in Lemma \ref{L5.2}.

If $D=\|A\|^2I_n$, we set $\beta^k=\beta/(\tau^k)^2$ for  $\beta>0$ and suppose it holds that\begin{equation}\label{G105}
\beta \|A\|^2/(\tau^k)^2+\sigma/\tau^{k}\geq\beta\|A\|^2/(\tau^{k+1})^2,~\sigma>0.
\end{equation}
Then, we have
\begin{small}
	$$
	\begin{pmatrix}
	\frac{\beta}{(\tau^{k})^2}(\|A\|^2I_n-A^TA)&0\\
	0&\frac{1}{\gamma\beta}I_l
	\end{pmatrix} +\begin{pmatrix}
	\frac{\sigma}{\tau^k}I_n&0\\0&0
	\end{pmatrix}\succeq\begin{pmatrix}
	\frac{\beta}{(\tau^{k+1})^2}(\|A\|^2I_n-A^TA)&0\\
	0&\frac{1}{\gamma\beta}I_l
	\end{pmatrix}.
	$$
\end{small}
Therefore, with the settings $H^k_0=\begin{pmatrix}
\frac{\beta}{(\tau^{k})^2}(\|A\|^2I_n-A^TA)&0\\
0&\frac{1}{\gamma\beta}I_l
\end{pmatrix}$, $r^k=1/\tau^k$, $z^k=\widetilde{x}^k$, $z'=x'$ and $\varTheta^{k+1}-\varTheta^{k}=(2-\gamma)\beta^k
\|A\widetilde{x}^k-b\|^2$, the convergence condition \eqref{V12} holds.

Then we can verify the conditions required in Lemma \ref{L5.2}.
For the two cases $D=A^TA$ and $D=\|A\|^2I_n$, we can infer $\lambda^{k+1}=\lambda^k$ by \eqref{G22} and \eqref{G23}.
\begin{rem}
	\label{r4}
	Consider \eqref{G85}. According  to \eqref{G88}, if $\bar{\lambda}^{k+1}$ satisfies \eqref{G23}, we can show that \eqref{G88} implies that
	\begin{equation}\nonumber
	\begin{aligned}
	&	\frac{1}{(\tau^k)^2}[S^{k+1}-\frac{(2-\gamma)\beta }{2(\tau^k)^2}\|A\breve{x}^{k}-b\|^2]-
	\frac{1}{(\tau^{k-1})^2}[S^{k}-\frac{(2-\gamma)\beta}{2(\tau^{k-1})^2}\|A\breve{x}^{k-1}-b\|^2]\\\geq&
	\frac{1}{2}\Big(\|v^{k+1}-v'\|^2_{H^{k+1}_0}-\|v^k-v'\|^2_{H^k_0} +
	\sigma\frac{1-\tau^k}{(\tau^k)^3}\|\breve{x}^{k}-\breve{x}^{k-1}\|^2
	\Big).
	\end{aligned}
	\end{equation} 	
	For the special setting $u'=u^*$, it leads to $O(1/k^4)$ convergence rate of $\min_{i}\|\breve{x}^{i}-\breve{x}^{i-1}\|^2$, $i=0,1,\dots,k$. This   improves the $O(1/k^3)$ convergence rate presented in \cite{zhang2022faster}.
\end{rem}

\subsection{Correcting multiplier once}
We consider updating $\{x^{k+1}\}$  in the following:
\begin{align}\label{G113}
\begin{cases}
\hat{\lambda}^k=\lambda^k-(1-\gamma)(1-\tau^k)\beta^k(A{x}^{k}-b),\\
\hat{x}^k=x^k+\frac{\tau^k(1-\tau^{k-1})}{\tau^{k-1}}(x^k-x^{k-1}),\\
x^{k+1}\in\arg\min\limits_x \{f(x)+x^T\nabla_x\varphi^k(\hat{x}^k,\hat{\lambda}^k)+\frac{\beta^k}{2}\|x-\hat{x}^k\|_D^2+\frac{\sigma(1-\tau^k)}{2\tau^k}\|x-x^k\|^2\}.
\end{cases}
\end{align}
We  define $v^k:=(\bar{x}^k,{\lambda}^k)$ with $\bar{x}^{k}$ in  \eqref{G12},  the artificial vectors $\breve{x}^k$ and  $\widetilde{v}^k:=(\widetilde{x}^k,\widetilde{\lambda}^k)$ as  in \eqref{G5}.  If $\lambda^{k+1}$ satisfies \eqref{G20} with $\gamma\in(0,2]$, by an analysis similar to that in subsection \ref{s4.2.1}, we can verify the  conditions \eqref{V9}-\eqref{V12}. Based on Lemma \ref{L5.2}, we can also establish $O(1/k^2)$ convergence rate.

\section{Applications}\label{S5}
In this section,
we present a few Lagrangian-based methods satisfying the prediction-correction framework given in sections \ref{subsubs5.1.1} and \ref{s4.2} for solving \eqref{P2} and \eqref{P3}. The convergence result follows from Lemmas \ref{L5.1},  \ref{L5.2} and  \ref{L5.3}. All the proofs of this section are given in Appendix.
\subsection{Applications in solving \eqref{P2}}\label{s4.3.1}
Throughout Subsection \ref{s4.3.1}, we assume that $f_2$ is $\sigma~(\geq0)$-strongly convex.
We define $u$, $v$ and $F(u)$ in  \eqref{V18},  the differentiable part of the augmented Lagrangian function of \eqref{P2}:
\begin{equation}
\label{G97}
\varphi^k(x_1,x_2,\lambda):=-\lambda^T(A_1x_1+A_2x_2-b)+\frac{\beta^k}{2}\|A_1x_1+A_2x_2-b\|^2,~\beta^k>0,
\end{equation}and let
\begin{equation}
\begin{aligned}
&Q^k=\begin{pmatrix}
\tau^k\beta^kD&0\\
-	A_2&\frac{1}{\tau^k\beta^k}I_l
\end{pmatrix}
,~M^k=\begin{pmatrix}
I_{n_2}&0\\
-\gamma\tau^k\beta^kA_2&\gamma I_l
\end{pmatrix},~\gamma\in (0,1],\\ &H^k=\begin{pmatrix}
\tau^k\beta^kD&0\\
0&\frac{1}{\gamma\tau^k\beta^k}I_l
\end{pmatrix},~ G^k= \begin{pmatrix}
\tau^k\beta^k(D-	\gamma A_2^TA_2)&-(1-\gamma)A_2^T\\
-(1-\gamma)A_2&\frac{2-\gamma}{\tau^k\beta^k}I_l
\end{pmatrix}. \label{G50}
\end{aligned}
\end{equation}

We first present the following algorithm for solving \eqref{P2}.
\begin{algo}\label{G40}{(Correcting multiplier twice.)}
	\begin{equation}\nonumber
	\begin{cases}
	\hat{x}^k=x^k+\frac{\tau^k(1-\tau^{k-1})}{\tau^{k-1}}(x^k-x^{k-1}),
	\\
	x^{k+1}_1\in\arg\min\limits_{x_1}\{f_1(x_1)+x_1^T\nabla_{x_1} \varphi^k(\hat{x}_1^k,\hat{x}_2^k,\lambda^k)+\frac{\beta^k}{2}\|x_1-\hat{x}_1^k\|_{A_1^TA_1}\},
	\\
	x_2^{k+1}\in\arg\min\limits_{x_2}\{f_2(x_2)+x_2^T\nabla_{x_2} \varphi^k({x}_1^{k+1},\hat{x}_2^{k},\lambda^k)+\frac{\beta^k}{2}\|x_2-\hat{x}_2^k\|_{D}\\~~~~~~~~~~+\frac{\sigma (1-\tau^k)}{2\tau^k}\|x_2-x_2^k\|^2_{D/{\sigma_{\rm max}(D)  }} \},
	\\	\lambda^{k+1}=\lambda^k.
	\end{cases}
	\end{equation}
\end{algo}
\begin{theorem}\label{T1}
	Let
	$v^k:=(\bar{x}_2^k,\bar{\lambda}^k)$,
	the artificial vectors $\breve{x}^k$ and  $\widetilde{v}^k:=(\widetilde{x}_2^k,\widetilde{\lambda}^k)$ with
	$\bar{x}^{k}$ given in \eqref{G12},
	\begin{eqnarray}\label{G45}
	&&\bar{\lambda}^k:=\lambda^k-\gamma(1-\tau^k)\beta^k(Ax^k-b),\\&&\widetilde{\lambda}^k:=\bar{\lambda}^k-\tau^k\beta^k(A_1\widetilde{x}_1^{k}+A_2\bar{x}_2^k-b),\label{G94}\\&&
	\breve{x}^k:=\begin{pmatrix}
	\breve{x}_1^k\\\breve{x}_2^k
	\end{pmatrix}=x^{k+1}=\begin{pmatrix}
	x_1^{k+1}\\x_2^{k+1}
	\end{pmatrix},~
	\widetilde{x}^k:=\begin{pmatrix}
	\widetilde{x}_1^k\\\widetilde{x}_2^k
	\end{pmatrix}=\bar{x}^{k+1}=\begin{pmatrix}
	\bar{x}_1^{k+1}\\\bar{x}_2^{k+1}
	\end{pmatrix}.\label{G46}
	\end{eqnarray}
	For $D=A_2^TA_2$ or $\|A_2\|^2I_{n_2}$, the following statements hold:	
	
	\noindent\textnormal{(\rmnum{1})}
	Suppose that $\beta^k=\beta/\tau^k$ ($\beta>0$) and the condition \eqref{G13} holds. Then
	Algorithm \ref{G40} satisfies the prediction-correction framework \eqref{G27}-\eqref{G28} with $c^k=(1-\gamma)(1-\tau^k)\beta^k$.  Moreover, the sequence $\{v^{k+1}\}$  satisfies the convergence conditions $\eqref{V9}$-$\eqref{V10}$ and $\eqref{V12}$ with $r^k=1$, $\sigma=0$,  $\varTheta^{k+1}-\varTheta^{k}\geq(1-\gamma)\tau^k\beta^k\|A\widetilde{x}^k-b\|^2$ and $H_0^k=H^k=\begin{pmatrix}
	\beta D &0\\
	0&\frac{1}{\gamma\beta}I_l
	\end{pmatrix}.$
	
	\noindent\textnormal{(\rmnum{2})}
	Suppose that $f_2$ is $\sigma~(>0)$-strongly convex, $\beta^k=\beta/(\tau^k)^2$ with	
	\begin{equation}\label{G56}
	\frac{1}{\tau^k}\left(
	\frac{\beta}{\tau^k}+\frac{\sigma}{\sigma_{\rm max}(D)}
	\right)\geq\frac{\beta}{(\tau^{k+1})^2},~\beta>0,
	\end{equation}
	and the condition \eqref{G29} holds.
	Then	Algorithm \ref{G40} satisfies the prediction-correction framework \eqref{G54}-\eqref{G55} with $c^k=(1-\gamma)(1-\tau^k)\beta^k/\tau^k$, $R={D/{\sigma_{\rm max}(D)  }} $, $z=x_2$ and   $z^k=\widetilde{x}_2^{k}$.  Moreover, the sequence $\{v^{k+1}\}$  satisfies the convergence conditions $\eqref{V9}$-$\eqref{V10}$ and $\eqref{V12}$ with $r^k=1/\tau^k$,  $\varTheta^{k+1}-\varTheta^{k}\geq(1-\gamma)\beta^k\|A\widetilde{x}^k-b\|^2$ and $H_0^k=\frac{1}{\tau^k}H^k=\begin{pmatrix}
	\beta/(\tau^{k})^2  D &0\\
	0&\frac{1}{\gamma\beta}I_l
	\end{pmatrix}.$
\end{theorem}

In the following, we provide different understandings of  Algorithm \ref{G40}.
\begin{theorem}\label{T4.5}	
	Let $\lambda^0=0$.  Algorithm \ref{G40} is equivalent to the following penalty method:
	\begin{align}\label{G58:2}
	\begin{cases}
	\hat{x}^k=x^k+\frac{\tau^k(1-\tau^{k-1})}{\tau^{k-1}}(x^k-x^{k-1}),\\
	x^{k+1}_1\in\arg\min\limits_{x_1}\{f_1(x_1)+\frac{\beta^k}{2}\|A_1x_1+A_2\hat{x}_2^k-b\|^2\},
	\\ x_2^{k+1}\in\arg\min\limits_{x_2}\{f_2(x_2)+\frac{\beta^k}{2}\|A_1x^{k+1}_1+A_2{x}_2-b\|^2+\frac{\beta^k}{2}\|x_2-\hat{x}_2^k\|^2_{D-A_2^TA_2}\\~~~~~~~~~~
	+\frac{\sigma (1-\tau^k)}{2\tau^k}\|x_2-x_2^k\|^2_{D/{\sigma_{\rm max}(D)  }}\}.
	\end{cases}
	\end{align}
\end{theorem}

The other  algorithm for solving \eqref{P2} is as follows.
\begin{algo}\label{G118}{(Correcting multiplier once.)}
	\begin{equation}\nonumber
	\begin{cases}
	\hat{\lambda}^k=\lambda^k+\gamma(1-\tau^k)\beta^k(Ax^k-b),~\gamma\in(0,1],\\
	\hat{x}^k=x^k+\frac{\tau^k(1-\tau^{k-1})}{\tau^{k-1}}(x^k-x^{k-1}),
	\\
	x^{k+1}_1\in\arg\min\limits_{x_1}\{f_1(x_1)+x_1^T\nabla_{x_1} \varphi^k(\hat{x}_1^k,\hat{x}_2^k,\hat{\lambda}^k)+\frac{\beta^k}{2}\|x_1-\hat{x}_1^k\|_{A_1^TA_1}\},
	\\
	x_2^{k+1}\in\arg\min\limits_{x_2}\{f_2(x_2)+x_2^T\nabla_{x_2} \varphi^k({x}_1^{k+1},\hat{x}_2^{k},\hat{\lambda}^k)+\frac{\beta^k}{2}\|x_2-\hat{x}_2^k\|_{D}\\~~~~~~~~~~+\frac{\sigma (1-\tau^k)}{2\tau^k}\|x_2-x_2^k\|^2_{D/{\sigma_{\rm max}(D)  }} \},
	\\	\lambda^{k+1}=\lambda^k-\gamma\tau^k\beta^k(A\bar{x}^{k+1}-b), ~\bar{x}^{k+1}~{\rm satisfies~\eqref{G12}}.
	\end{cases}
	\end{equation}
\end{algo}
\begin{theorem}\label{T3}
	Let $Q^k,~M^k,~H^k$ and $G^k$ be given in \eqref{G50},
	$v^k:=(\bar{x}_2^k,{\lambda}^k)$ with
	$\bar{x}^{k}$ in \eqref{G12},
	the artificial vectors $\breve{x}^k$ and  $\widetilde{v}^k:=(\widetilde{x}_2^k,\widetilde{\lambda}^k)$ with $\breve{x}^k$ and $\widetilde{x}^k$ in \eqref{G46}, and
	\begin{eqnarray}
	&&\widetilde{\lambda}^k:={\lambda}^k-\tau^k\beta^k(A_1\widetilde{x}_1^{k}+A_2\bar{x}_2^k-b).\label{G115}
	\end{eqnarray}
	For $D=A_2^TA_2$ or $\|A_2\|^2I_{n_2}$,
	the statements \textnormal{(\rmnum{1})} and \textnormal{(\rmnum{2})} in Theorem \ref{T1} hold for
	Algorithm \ref{G118}.
\end{theorem}

\begin{rem}
	Note that Tran-Dinh  and  Zhu  \cite{2020NON} also considered the primal-dual algorithms for solving  \eqref{P2} with a non-ergodic convergence rate in a similar iterative format.
	However, in order to achieve the non-ergodic convergence in the strongly convex case, it is at the cost of evaluating the proximal operator of $f_2$   twice per iteration \cite{2020NON}.
\end{rem}

\subsection{Applications in solving \eqref{P3}}\label{s4.3.2}
With the help of our prediction-correction framework, the algorithms with the ergodic convergence in terms of the primal-dual gap \cite{he2017convergence}  can be rebuilt to achieve the non-ergodic convergence rate for solving \eqref{P3}.

Throughout subsection \ref{s4.3.2}, we define $u$, $v$ and $F(u)$ in  \eqref{D13}, rewrite the differentiable part of the augmented Lagrangian function of \eqref{P3} as
\begin{equation}
\label{G98}
\varphi^k(x_1,\cdots, x_m,\lambda):=-\lambda^T\Big(\sum_{i=1}^{m}A_ix_i-b\Big)+\frac{\beta^k}{2}\Big\|\sum_{i=1}^{m}A_ix_i-b\Big\|^2,\beta^k>0,
\end{equation}
and let
\begin{equation} \label{G73}
\begin{aligned}
&J=\begin{pmatrix}
I_{l}&0&\cdots&0\\
I _{l}&I _{l}&\cdots&    0    \\
\vdots&\ddots&\ddots&\vdots\\
I_{l}&\cdots&I _{l}&I _{l}\\
\end{pmatrix}\in\mathbb{R}^{(m-1)l\times (m-1)l},~
\widetilde{I}=\begin{pmatrix}
I_l&\cdots &I_l
\end{pmatrix}\in\mathbb{R}^{l\times (m-1)l},\\
&
P^k=\begin{pmatrix}
\sqrt{\tau^k\beta^k}J&0\\
0&\frac{1}{\sqrt{\tau^k\beta^k}}I_l
\end{pmatrix}
,~
N^k=\gamma \begin{pmatrix}
\sqrt{\tau^k\beta^k}I_{(m-1)l}&0\\
-\sqrt{\tau^k\beta^k}\widetilde{I}&\frac{1}{\sqrt{\tau^k\beta^k}}I_l
\end{pmatrix},~\gamma\in(0,1],\\&
Q^k=\begin{pmatrix}
{\tau^k\beta^k}J&0\\
-\widetilde{I}&\frac{1}{{\tau^k\beta^k}}I_l
\end{pmatrix},~
M^k=(P^k)^{-T}N^k,~\bar{J}=\begin{pmatrix}
0_{n_1}&0\\
0&J^{-T}
\end{pmatrix}.
\end{aligned}
\end{equation}

We present two algorithms for solving \eqref{P3} and then establish the  non-ergodic convergence rates by showing that the conditions required in Lemmas \ref{L5.1} and  \ref{L5.3} are all satisfied.
\begin{algo}\label{G63}{(Correcting multiplier twice.)}
	\begin{equation}\nonumber
	\begin{cases}
	\hat{x}^k=(1-\tau^k)\breve{x}^{k-1}+\tau^k\bar{x}^k,  \\
	\breve{x}_1^k=\arg\min\limits_{x_1}\{ f_1(x_1)+x_1^T\nabla_{x_1}{\varphi}^k(\hat{x}_{1}^k,\cdots,\hat{x}_m^k,\lambda^k)
	+\frac{\beta^k}{2}\|x_1-\hat{x}_1^k\|_{A_1^TA_1}^2\},\\
	\breve{x}_j^k=\arg\min\limits_{x_j}\{ f_j(x_j)+x_j^T\nabla_{x_j}{\varphi}^k(\breve{x}_1^k,\cdots,\breve{x}_{j-1}^k,\hat{x}_{j}^k,\cdots,\hat{x}_m^k,\lambda^k)\\
	~~~~~~
	+\frac{\beta^k}{2}\|x_j-\hat{x}_j^k\|_{A_j^TA_j}^2\},~j=2,\cdots,m,\\
	\widetilde{x}^k=\breve{x}^{k}/\tau^k-(1-\tau^k)\breve{x}^{k-1}/\tau^k,\\
	\bar{x}^{k+1}=\bar{x}^k-\gamma \bar{J}(\bar{x}^k-\widetilde{x}^k),~\gamma\in(0,1],\\
	\lambda^{k+1}=\lambda^k-\gamma[ (1-\tau^{k})\beta^k(A\breve{x}^{k-1}-b)-(1-\tau^{k+1})\beta^{k+1}(A\breve{x}^{k}-b) ]-\gamma\tau^k\beta^k(A\widetilde{x}^{k}-b).
	
	\end{cases}
	\end{equation}
\end{algo}

\begin{theorem}\label{T4.6}
	Let $P^k,N^k,~Q^k,~M^k$ and $ \bar{J}$ be given in \eqref{G73}. Define \begin{equation}
	\label{G72}
	\breve{x}^k=\begin{pmatrix}
	\breve{x}_1^k\\\breve{x}^k_2\\\vdots\\\breve{x}^k_m
	\end{pmatrix},~
	\widetilde{u}^k=\begin{pmatrix}
	\widetilde{x}_1^k\\\vdots\\\widetilde{x}^k_m\\\widetilde{\lambda}^k
	\end{pmatrix},~v^k=\begin{pmatrix}
	A_2\bar{x}^k_2\\\vdots\\A_m\bar{x}^k_m\\\bar{\lambda}^k
	\end{pmatrix},~\widetilde{v}^k=\begin{pmatrix}
	A_2\widetilde{x}^k_2\\\vdots\\A_m\widetilde{x}^k_m\\\widetilde{\lambda}^k
	\end{pmatrix},
	\end{equation}
	with \begin{eqnarray}
	&&
	\bar{\lambda}^k:= \lambda^k-\gamma (1-\tau^k)\beta^k(A\breve{x}^{k-1}-b),\label{G95}\\&&
	\widetilde{\lambda}^k:=\bar{\lambda}^k-\tau^k\beta^k	(A_1\widetilde{x}_1^k+\sum_{j=2}^{m}A_j\bar{x}_j^k-b).\label{G104}
	\end{eqnarray}
	The following statements hold:
	
	\noindent\textnormal{(\rmnum{1})}
	If $\beta^k=\beta/\tau^k$ for some $\beta>0$ and the  condition \eqref{G13} holds, then Algorithm \ref{G63} satisfies the prediction-correction framework \eqref{G27}-\eqref{G28} with $c^k=(1-\gamma)(1-\tau^k)\beta^k$. Moreover, the sequence $\{v^{k+1}\}$  satisfies the convergence conditions $\eqref{V9}$-$\eqref{V10}$ and $\eqref{V12}$ with $r^k=1$, $\sigma=0$,  $\varTheta^{k+1}-\varTheta^{k}\geq(1-\gamma)\tau^k\beta^k\|A\widetilde{x}^k-b\|^2$ and  $H^{k}_0=\frac{1}{\gamma}\begin{pmatrix}
	\beta JJ^T&0\\
	0&\frac{1}{\beta}I_l
	\end{pmatrix}$.
	
	\noindent\textnormal{(\rmnum{2})}
	If  $f_m$ is $L$-gradient Lipschitz continuous, $\beta^k\equiv\beta>0$, $(1-\gamma)\beta\leq 1$,
		\begin{equation}\label{G79}
		\frac{1}{\beta(\tau^k)^2}+\frac{\sigma''}{\tau^k}\geq	\frac{1}{\beta(\tau^{k+1})^2}+\frac{\sigma''(1-\gamma)}{\tau^{k+1}},
		~\sigma''=\sigma'-\frac{\sigma'^2}{\sigma'+1},~\sigma'=\frac{\sigma_{\rm min}(A_mA_m^T)}{L},
		\end{equation}
	and the condition \eqref{G29} holds,
	then Algorithm \ref{G63} satisfies the prediction-correction framework \eqref{G54}-\eqref{G55} with $c^k=(1-\gamma)(1-\tau^k)\beta^k/\tau^k$, $\sigma=1/L$, $R=I_{n_m}$, $z=\nabla f_m({x}_m)$ and  $z^k=\nabla_{} f_m(\breve{x}_m^k)$.  Moreover, the sequence $\{v^{k+1}\}$  satisfies the convergence conditions $\eqref{V9}$-$\eqref{V10}$ and $\eqref{V12}$ with $r^k=1/\tau^k$,  $\varTheta^{k+1}-\varTheta^{k}\geq(1-\gamma)\beta^k(\|A\widetilde{x}^k-b\|^2
	-  \frac{\tau^k}{(\tau^{k-1})^2}\|A\breve{x}^{k-1}-b\|^2
	)$ and  $H^{k}_0=\frac{1}{\gamma}\begin{pmatrix}
	\beta JJ^T&0\\
	0&\frac{1}{\beta(\tau^k)^2}I_l+\frac{\sigma''(1-\gamma)}{\tau^k}I_l
	\end{pmatrix}  $, the special setting $z'=\nabla f_m({x}_m^*)$ and $v=v^*$.
\end{theorem}

We can establish the equivalence between Algorithm \ref{G63} and the penalty method.
\begin{theorem}\label{T4.8}
	Let $\lambda^0=0$.  Algorithm \ref{G63} with $\beta^k=\beta/\tau^k$ is equivalent to the following penalty method:
	\begin{equation}\nonumber
	\begin{cases}	
	\hat{x}^k=(1-\tau^k)\breve{x}^{k-1}+\tau^k\bar{x}^k,  \\
	\breve{x}^{k}_1\in\arg\min\limits_{x_1}\{f_1(x_1)+\frac{\beta^k}{2}\|A_1x_1+\sum_{i=2}^{m}A_i\hat{x}_i^k-b\|^2\},\\
	\breve{x}^{k}_j\in\arg\min\limits_{x_{j, j=2,\cdots,m} }\{f_j(x_j)+\frac{\beta^k}{2}\|\sum_{i=1}^{j-1}A_i\breve{x}_i^k+A_jx_j+\sum_{i=j+1}^{m}A_i\hat{x}_i^k-b\|^2\},\\
	\widetilde{x}^k=\breve{x}^{k}/\tau^k-(1-\tau^k)\breve{x}^{k-1}/\tau^k,\\
	\bar{x}^{k+1}=\bar{x}^k-\gamma \bar{J}(\bar{x}^k-\widetilde{x}^k),~\gamma\in(0,1].	
	\end{cases}
	\end{equation}
\end{theorem}

The other algorithm for solving \eqref{P3} is as follows.

\begin{algo}\label{G117}{(Correcting multiplier once.)}
	\begin{equation}\nonumber
	\begin{cases}
	\hat{\lambda}^k=\lambda^k+\gamma(1-\tau^k)\beta^k(A\breve{x}^{k-1}-b),
	\\
	\hat{x}^k=(1-\tau^k)\breve{x}^{k-1}+\tau^k\bar{x}^k,  \\
	\breve{x}_1^k=\arg\min\limits_{x_1}\{ f_1(x_1)+x_1^T\nabla_{x_1}{\varphi}^k(\hat{x}_{1}^k,\cdots,\hat{x}_m^k,\hat{\lambda}^k)
	+\frac{\beta^k}{2}\|x_1-\hat{x}_1^k\|_{A_1^TA_1}^2\},\\
	\breve{x}_j^k=\arg\min\limits_{x_j}\{ f_j(x_j)+x_j^T\nabla_{x_j}{\varphi}^k(\breve{x}_1^k,\cdots,\breve{x}_{j-1}^k,\hat{x}_{j}^k,\cdots,\hat{x}_m^k,\hat{\lambda}^k)\\
	~~~~~~
	+\frac{\beta^k}{2}\|x_j-\hat{x}_j^k\|_{A_j^TA_j}^2\},~j=2,\cdots,m,\\
	\widetilde{x}^k=\breve{x}^{k}/\tau^k-(1-\tau^k)\breve{x}^{k-1}/\tau^k,\\
	\bar{x}^{k+1}=\bar{x}^k-\gamma \bar{J}(\bar{x}^k-\widetilde{x}^k),~\gamma\in(0,1],\\
	\lambda^{k+1}=\lambda^k-\gamma\tau^k\beta^k(A\widetilde{x}^{k}-b).
	
	\end{cases}
	\end{equation}
\end{algo}

\begin{theorem}\label{T4.7}
	Define $P^k,N^k,~Q^k,~M^k$ and $ \bar{J}$ in \eqref{G73}, $\breve{x}^k$,  $\widetilde{u}^k$ and $\widetilde{v}^k$ in \eqref{G72}
	with \begin{eqnarray}
	&&
	\widetilde{\lambda}^k:={\lambda}^k-\tau^k\beta^k	(A_1\widetilde{x}_1^k+\sum_{j=2}^{m}A_j\bar{x}_j^k-b).
	\end{eqnarray}
	Let	$v^k=\begin{pmatrix}
	A_2\bar{x}_2^k,&\cdots,&A_m\bar{x}_m^k,&\lambda^k
	\end{pmatrix}$. Then the statements \textnormal{(\rmnum{1})} and \textnormal{(\rmnum{2})} in Theorem \ref{T4.6} hold for Algorithm \ref{G117}.
\end{theorem}

\appendix
\section{Proofs in Section \ref{S5}}
\begin{proof}({\bfseries Proof of Theorem \ref{T1}})
	\noindent\textnormal{(\rmnum{1})}
	We first write the optimality condition of the $x_1$-subproblem in Algorithm \ref{G40} as
	\begin{equation}\label{G41}
	\begin{aligned}
	&f_1(x_1)-f_1(x^{k+1}_1)+(x_1-x_1^{k+1})^T[
	-A_1^T\lambda^k\\&+\beta^kA_1^T(A_1{x}_1^{k+1}+A_2\hat{x}_2^k-b)]\geq0,~\forall x_1.	
	\end{aligned}
	\end{equation}
	Since it holds that
	\begin{equation}\label{G70}
	\begin{aligned}
	&	-A_1^T\lambda^k+\beta^kA_1^T(A_1{x}_1^{k+1}+A_2\hat{x}_2^k-b)\\=&-A_1^T\lambda^k+\gamma (1-\tau^k)\beta^kA_1^T(Ax^k-b) +\tau^k\beta^kA_1^T(A_1\bar{x}_1^{k+1}\\&+A_2\bar{x}_2^k-b)+(1-\gamma)(1-\tau^k)\beta^kA_1^T(Ax^k-b)\\=&
	-A_1^T\widetilde{\lambda}^k+(1-\gamma)(1-\tau^k)\beta^kA_1^T(A\breve{x}^{k-1}-b),
	\end{aligned}
	\end{equation}
	multiplying both sides of \eqref{G41} by $(1-\tau^k)/\tau^k$ with  $x=\breve{x}^{k-1}$, and then adding it to   \eqref{G41} yields that
	\begin{equation}\label{G42}
	\begin{aligned}
	&\frac{1}{\tau^k}[f_1(x_1)-f_1(\breve{x}_1^{k})]-\frac{1}{\tau^{k-1}}[f_1(x_1)-f_1(\breve{x}_1^{k-1})]\\
	+&(x_1-\widetilde{x}_1^{k})^T\{-A_1^T\widetilde{\lambda}^k+(1-\gamma)(1-\tau^k)\beta^kA_1^T(A\breve{x}^{k-1}-b)\}\geq0,~\forall x_1.
	\end{aligned}
	\end{equation}
	Let $f_2$ be $\sigma$-strongly convex with $\sigma\ge 0$. 	
	The optimality condition of the $x_2$-subproblem in Algorithm \ref{G40} can be written as
	\begin{equation}\label{G91}
	\begin{aligned}
	&f_2(x_2)-f_2(x^{k+1}_2)+(x_2-x_2^{k+1})^T[-A_2^T\lambda^k+\beta^kA_2^T(A_1{x}_1^{k+1}+A_2\hat{x}_2^k-b)\\&+\beta^kD(x_2^{k+1}-\hat{x}_2^k)+\frac{\sigma(1-\tau^k)}{\sigma_{{\rm max}}(D)\tau^k}D(x_2^{k+1}-x_2^k)]\\\geq&\frac{\sigma}{2}\|x_2^{k+1}-x_2\|^2\geq\frac{\sigma}{2}\|x_2^{k+1}-x_2\|^2_{D/{\sigma_{\rm max}(D)  }}  ,~\forall x_2.	
	\end{aligned}
	\end{equation}
	Note that
	\begin{equation}\label{G71}
	\begin{aligned}
	&-A_2^T\lambda^k+\beta^kA_2^T(A_1{x}_1^{k+1}+A_2\hat{x}_2^k-b)+\beta^kD(x_2^{k+1}-\hat{x}_2^k)\\=&
	-A_2^T\lambda^k+\gamma (1-\tau^k)\beta^kA_2^T(Ax^k-b) +\tau^k\beta^kA_2^T(A_1\bar{x}_1^{k+1}+A_2\bar{x}_2^k-b)
	\\&+\tau^k\beta^kD(\bar{x}_2^{k+1}-\bar{x}_2^k)
	+(1-\gamma)(1-\tau^k)\beta^kA_2^T(Ax^k-b)\\=&
	-A_2^T\widetilde{\lambda}^k+\tau^k\beta^kD(\widetilde{x}_2^{k}-\bar{x}_2^k)
	+(1-\gamma)(1-\tau^k)\beta^kA_2^T(A\breve{x}^{k-1}-b).
	\end{aligned}
	\end{equation}
	Similar to getting \eqref{G42}, it follows from \eqref{G91} with  $\sigma=0$ that
	\begin{equation}\label{G44}
	\begin{aligned}
	&\frac{1}{\tau^k}[f_2(x_2)-f_2(\breve{x}_2^{k})]-\frac{1}{\tau^{k-1}}[f_2(x_2)-f_2(\breve{x}_2^{k-1})]
	+(x_2-\widetilde{x}_2^{k})^T[-A_2^T\widetilde{\lambda}^k\\&+\tau^k\beta^kD(\widetilde{x}_2^{k}-\bar{x}_2^k)
	+(1-\gamma)(1-\tau^k)\beta^kA_2^T(A\breve{x}^{k-1}-b)]\geq 0,
	~\forall x_2.
	\end{aligned}
	\end{equation}
	According to the definition $\widetilde{\lambda}^k$ \eqref{G94}, we have
	\begin{equation}\label{G47}
	(\lambda-\widetilde{\lambda}^k)^T[
	(A_1\widetilde{x}_1^k+A_2\widetilde{x}_2^k-b)-A_2(\widetilde{x}^k_2-\bar{x}^k_2)-\frac{1}{\tau^k\beta^k}( \bar{\lambda}^k- \widetilde{\lambda}^k)	
	]\geq0,~\forall \lambda.
	\end{equation}
	Based on the definitions of $u$, $v$ and $F(u)$ in \eqref{V18}, we can merge the inequalities \eqref{G42}, \eqref{G44} and \eqref{G47} into the following one:
	\begin{equation}\label{G48}
	\begin{aligned}
	&\frac{1}{\tau^k}[f(x)-f(\breve{x}^{k})]-\frac{1}{\tau^{k-1}}[f(x)-f(\breve{x}^{k-1})]+(u-\widetilde{u}^k)^TF(\widetilde{u}^k)\\&+(1-\gamma)(1-\tau^k)\beta^k(A(x-\widetilde{x}^k))^T(A\breve{x}^{k-1}-b)\\
	\geq&
	\begin{pmatrix}
	x_2-\widetilde{x}_2^k\\\lambda-\widetilde{\lambda}^k
	\end{pmatrix}^T\begin{pmatrix}
	\tau^k\beta^kD&0\\
	-A_2&\frac{1}{\tau^k\beta^k}I_l
	\end{pmatrix}
	\begin{pmatrix}
	\bar{x}_2^k-\widetilde{x}_2^k\\\bar{\lambda}^k-\widetilde{\lambda}^k
	\end{pmatrix},~\forall u.	
	\end{aligned}
	\end{equation}
	Then based on the definition of $\bar{x}^k$ in \eqref{G12}, we have
	\begin{equation}\label{G114}
	\begin{aligned}
	\lambda^{k+1}=&\lambda^k-\gamma[ (1-\tau^{k})\beta^k(Ax^{k}-b)-(1-\tau^{k+1})\beta^{k+1}(Ax^{k+1}-b) ]\\&-\gamma\tau^k\beta^k(A\bar{x}^{k+1}-b),~\gamma\in(0,1].
	\end{aligned}
	\end{equation}
	Hence, it holds that
	\begin{eqnarray*}
		\bar{\lambda}^{k+1}&\overset{\eqref{G45}}{=}&{\lambda}^{k+1}-\gamma(1-\tau^{k+1})\beta^{k+1}(A{x}^{k+1}-b)\\&\overset{\eqref{G114}}{=}&\bar{\lambda}^k-\gamma\tau^k\beta^k(A\bar{x}^{k+1}-b)=\bar{\lambda}^k-\gamma(\bar{\lambda}^k- \widetilde{\lambda}^k)-\gamma\tau^k\beta^kA_2(\widetilde{x}_2^{k}-\bar{x}_2^k) .
	\end{eqnarray*}
	Then  we have
	\begin{equation}\label{G51}
	\begin{pmatrix}
	\bar{x}_2^{k+1}\\\bar{\lambda}^{k+1}
	\end{pmatrix}=\begin{pmatrix}
	\bar{x}^{k}_2\\\bar{\lambda}^{k}
	\end{pmatrix}-\begin{pmatrix}
	I_{n_2}&0\\
	-	\gamma\tau^k\beta^kA_2&\gamma I_l
	\end{pmatrix}\begin{pmatrix}
	\bar{x}_2^{k}-\widetilde{x}_2^k\\\bar{\lambda}^k-\widetilde{\lambda}^k
	\end{pmatrix}.
	\end{equation}
	Clearly, \eqref{G48} and \eqref{G51} satisfy the prediction-correction framework \eqref{G27}-\eqref{G28} with $c^k=(1-\gamma)(1-\tau^k)\beta^k$.
	
	According to the definitions of $Q^k,~M^k,~H^k$ and $G^k$  given in \eqref{G50}, we can verify the conditions $\eqref{V9}$-$\eqref{V10}$.  For $D=A_2^TA_2$ or $\|A_2\|^2I_{n_2}$, we obtain
	\begin{equation}\label{G52}
	\begin{aligned}
	&\|v^k-\widetilde{v}^k\|^2_{G^k}=\tau^k\beta^k\|\bar{x}_2^k-\bar{x}_2^{k+1}\|^2_{D}+(2-\gamma)\tau^k\beta^k\|A\bar{x}^{k+1}-b\|^2\\&+2\tau^k\beta^k(\bar{x}_2^k-\bar{x}_2^{k+1})^TA_2^T(A\bar{x}^{k+1}-b)
	\geq(1-\gamma)\tau^k\beta^k\|A\widetilde{x}^k-b\|^2,
	\end{aligned}
	\end{equation}
	where the last relation follows from  Cauchy-Schwarz inequality. We have verified the convergence conditions $\eqref{V9}$-$\eqref{V10}$ and $\eqref{V12}$.
	
	\noindent\textnormal{(\rmnum{2})}
	According to \eqref{G71},
	multiplying both sides of \eqref{G91}
	by $(1-\tau^k)/(\tau^k)^2$ and $1/\tau^k$, respectively, fixing the former at  $x=\breve{x}^{k-1}$,  and  then adding both together yields that
		\begin{equation}\label{G92}
		\begin{aligned}
		&\frac{1}{(\tau^k)^2}[f_2(x_2)-f_2(\breve{x}_2^{k})]-\frac{1}{(\tau^{k-1})^2}[f_2(x_2)-f_2(\breve{x}_2^{k-1})]
		+\frac{1}{\tau^k}(x_2-\widetilde{x}_2^{k})^T\{-A_2^T\widetilde{\lambda}^k+\\&\tau^k\beta^kD(\widetilde{x}_2^{k}-\bar{x}_2^k)
		+(1-\gamma)(1-\tau^k)\beta^kA_2^T(A\breve{x}^{k-1}-b)
		+\frac{\sigma(1-\tau^k)}{\sigma_{{\rm max}}(D)\tau^k}D(x_2^{k+1}-x_2^k)\}\\&\geq
		\frac{\sigma}{2}\Big(
		\frac{1-\tau^k}{(\tau^k)^2}\|\breve{x}_2^{k}-\breve{x}_2^{k-1}\|^2_{D/{\sigma_{\rm max}(D)  }}  +\frac{1}{\tau^k}\|\breve{x}_2^{k}-x_2\|^2_{D/{\sigma_{\rm max}(D)  }}
		\Big),
		~\forall x_2.
		\end{aligned}
		\end{equation}
	Similar to getting \eqref{G88}, we can show that \eqref{G92} is equivalent to
	\begin{equation}\label{G93}
	\begin{aligned}
	&\frac{1}{(\tau^k)^2}[f_2(x_2)-f_2(\breve{x}_2^{k})]-\frac{1}{(\tau^{k-1})^2}[f_2(x_2)-f_2(\breve{x}_2^{k-1})]
	+\frac{1}{\tau^k}(x_2-\widetilde{x}_2^{k})^T[-A_2^T\widetilde{\lambda}^k\\&+\tau^k\beta^kD(\widetilde{x}_2^{k}-\bar{x}_2^k)
	+(1-\gamma)(1-\tau^k)\beta^kA_2^T(A\breve{x}^{k-1}-b)
	]\\&
	\geq
	\frac{\sigma}{2}\Big(\frac{1}{\tau^k}\|\widetilde{x}_2^k-x_2\|^2 _{D/{\sigma_{\rm max}(D)  }}+
	\frac{1-\tau^k}{(\tau^k)^3}\|\breve{x}_2^{k}-\breve{x}_2^{k-1}\|^2 _{D/{\sigma_{\rm max}(D)  }}\Big),~\forall x_2.
	\end{aligned}
	\end{equation}
	By combining the $x_1$-subproblem, \eqref{G93} and \eqref{G47},  we obtain
	\begin{equation}\label{G53}
	\begin{aligned}
	&\frac{1}{(\tau^k)^2}[f(x)-f(\breve{x}^{k})]-\frac{1}{(\tau^{k-1})^2}[f(x)-f(\breve{x}^{k-1})]+\frac{1}{\tau^k}(u-\widetilde{u}^k)^TF(\widetilde{u}^k)\\&+(1-\gamma)(1-\tau^k)\beta^k\frac{1}{\tau^k}(A(x-\widetilde{x}^k))^T(A\breve{x}^{k-1}-b)\\\geq&\frac{1}{\tau^k}
	(v-\widetilde{v}^k)^T\begin{pmatrix}
	\tau^k\beta^kD&0\\
	-A_2&\frac{1}{\tau^k\beta^k}I
	\end{pmatrix}(v^k-\widetilde{v}^k)+\frac{\sigma}{2\tau^k}\|\widetilde{x}_2^{k}-x_2\|^2_{D/{\sigma_{\rm max}(D)  }},~\forall u.
	\end{aligned}
	\end{equation}
	We can also  verify that the correction steps \eqref{G51} and \eqref{G52} hold.  Therefore, according to \eqref{G53} and \eqref{G51},  the prediction-correction framework \eqref{G54}-\eqref{G55} holds with $c^k=(1-\gamma)(1-\tau^k)\beta^k/\tau^k$, $R={D/{\sigma_{\rm max}(D)  }}$ and $z^k=\widetilde{x}^k_2$.
	
	According  to the condition \eqref{G56} and the structure of $H^k$, we have
	\begin{eqnarray*}
		&& \frac{1}{\tau^k}(\frac{\beta}{\tau^k}+\frac{ \sigma}{\sigma_{\rm max}(D)})D\succeq \frac{\beta}{(\tau^{k+1})^2}D
		\Longrightarrow
		\frac{1}{\tau^k}H^k+\frac{1}{\tau^k}\begin{pmatrix}
			\frac{\sigma}{\sigma_{\rm max}(D)  } D&0\\
			0&0
		\end{pmatrix}\succeq H_0^{k+1}.
	\end{eqnarray*}
	It	then holds that
	$$
	\begin{aligned}
	&\frac{1}{\tau^k}
	(v'-\widetilde{v}^k)^TQ^k(v^k-\widetilde{v}^k)+\frac{\sigma}{2\tau^k}\|\widetilde{x}_2^{k}-x_2'\|^2_{D/{\sigma_{\rm max}(D)  }}
	\\\geq
	&\frac{1}{2\tau^k}\Big(
	\|v^{k+1}-v'\|^2_{H^k}+\sigma\|\widetilde{x}_2^{k}-x_2'\|^2_{D/{\sigma_{\rm max}(D)  }}-\|v^k-v'\|^2_{H^k}+\|\widetilde{v}^k-v^k\|^2_{G^k}
	\Big)\\\geq&
	\frac{1}{2}\left(\|v^{k+1}-v'\|^2_{H^{k+1}_0}-\|v^k-v'\|^2_{H^k_0}+(1-\gamma)\beta^k\|A\widetilde{x}^k-b\|^2\right),
	\end{aligned}
	$$
	which completes	the proof. 	
\end{proof}
\begin{proof}({\bfseries Proof of Theorem \ref{T4.5}})
	Since $\lambda^0=0$,
	in Algorithm \ref{G40}, we always have
	\begin{equation}\nonumber
	\begin{aligned}
	\lambda^{k+1}=\lambda^k=\dots=\lambda^0=0,~\forall k.
	\end{aligned}
	\end{equation}
	The following proof is based on the optimality conditions.
\end{proof}

\begin{proof}	({\bfseries Proof of Theorem \ref{T3}})
	Based on the optimality conditions similar to \eqref{G48} and \eqref{G53}, we can verify the prediction step. Next, according to
	$$\lambda^{k+1}={\lambda}^k-\gamma\tau^k\beta^k(A\bar{x}^{k+1}-b)={\lambda}^k-\gamma({\lambda}^k- \widetilde{\lambda}^k)-\gamma\tau^k\beta^kA_2(\widetilde{x}_2^{k}-\bar{x}_2^k) ,$$ we can verify the correction step. The proof of verifying the conditions \eqref{V9}-\eqref{V10} and \eqref{V12} is similar to that of Theorem \ref{T1}.
\end{proof}

\begin{proof}
	({\bfseries Proof of Theorem \ref{T4.6}})
	
	\noindent	\textnormal{(\rmnum{1})}
	For any $x_j$ ($j=1,\cdots,m$), the optimality condition of the $x_j$-subproblem in Algorithm \ref{G63} reads as
	\begin{equation}\label{G64}
	\begin{aligned} &f_j(x_j)-f_j(\breve{x}_j^k)+(x_j-\breve{x}^k_j)^T(-A_j^T\lambda^k+\beta^kA_j^T(\sum_{i=1}^{j}A_i\breve{x}_i^k+\sum_{i=j+1}^{m}A_i\hat{x}_i^k-b))\geq 0.
	\end{aligned}
	\end{equation}
	Since it holds that
		\begin{equation}\label{G65}
		\begin{aligned}
		&\lambda^k-\beta^k(\sum_{i=1}^{j}A_i\breve{x}_i^k+\sum_{i=j+1}^{m}A_i\hat{x}_i^k-b)\\=&
		\lambda^k-\gamma (1-\tau^k)\beta^k(A\breve{x}^{k-1}-b) -\tau^k\beta^k(A_1\widetilde{x}_1^{k}+\sum_{i=2}^mA_i\bar{x}_i^{k}-b)\\
		&-\tau^k\beta^k\sum_{i=2}^{j}A_i(\widetilde{x}_i^k-\bar{x}_i^{k})-(1-\gamma)(1-\tau^k)\beta^k(A\breve{x}^{k-1}-b)
		\\=&\widetilde{\lambda}^k-\tau^k\beta^k\sum_{i=2}^{j}A_i(\widetilde{x}_i^k-\bar{x}_i^{k})-(1-\gamma)(1-\tau^k)\beta^k(A\breve{x}^{k-1}-b),
		\end{aligned}
		\end{equation}
	multiplying both sides of \eqref{G64} by $(1-\tau^k)/\tau^k$  at  $x=\breve{x}^{k-1}$ and then adding it to \eqref{G64} yields that 
	\begin{equation}\label{G78}
	\begin{aligned}
	&	\frac{1}{\tau^k}[f_j(x_j)-f_j(\breve{x}_j^{k})]-\frac{1}{\tau^{k-1}}[f_j(x_j)-f_j(\breve{x}_j^{k-1})]+(x_j-\widetilde{x}^k_j)^T
	[-A_j^T\widetilde{\lambda}^k\\+&\tau^k\beta^k\sum_{i=2}^{j}A_j^TA_i(\widetilde{x}_i^k-\bar{x}_i^{k})+(1-\gamma)(1-\tau^k)\beta^kA_j^T(A\breve{x}^{k-1}-b)]\geq 0,~\forall x_j.
	\end{aligned}
	\end{equation}
	By the definition of $\widetilde{\lambda}^k$, we have
	\begin{equation}\label{G66}
	(\lambda-\widetilde{\lambda}^k)^T[
	(\sum_{i=1}^{m}A_i\widetilde{x}_i^k-b)-\sum_{i=2}^{m}A_i(\widetilde{x}_i^k-\bar{x}_i^k)-\frac{1}{\tau^k\beta^k}( \bar{\lambda}^k- \widetilde{\lambda}^k)	
	]\geq0,~\forall \lambda.
	\end{equation}
	Then the prediction step holds, since it follows	from  \eqref{G78} and \eqref{G66} that
	\begin{equation}\label{G82}
	\begin{aligned}
	&\frac{1}{\tau^k}[f(x)-f(\breve{x}^{k})]-
	\frac{1}{\tau^{k-1}}[f(x)-f(\breve{x}^{k-1})]+(u-\widetilde{u}^k)^TF(\widetilde{u}^k)
	\\&+c_1^k(A(x-\widetilde{x}^k))^T(A\breve{x}^{k-1}-b)\geq
	(v-\widetilde{v}^k)^TQ^k(v^k-\widetilde{v}^k), ~\forall u,
	\end{aligned}
	\end{equation}
	where $c_1^k=(1-\gamma)(1-\tau^k)\beta^k$ and $Q^k$ is defined in \eqref{G73}.

	Then according to the definitions of $\bar{\lambda}^k$ in \eqref{G95} and $\widetilde{\lambda}^k$ in \eqref{G104}, we have
	\begin{eqnarray} \bar{\lambda}^{k+1}&\overset{\eqref{G95}}{=}&
	{\lambda}^{k+1}-\gamma (1-\tau^{k+1})\beta^{k+1}(A\breve{x}^{k}-b)\nonumber\\			&\overset{{\rm Algorithm}~ \ref{G63}}{=}&\bar{\lambda}^k-\gamma\tau^k\beta^k(A\widetilde{x}^{k}-b)\label{G96}
	\\&=&\bar{\lambda}^k-\gamma\tau^k\beta^k\left(\sum_{i=2}^{m}A_i(\widetilde{x}_i^k-\bar{x}_i^k)+A_1\widetilde{x}_1^k+\sum_{i=2}^{m}A_i\bar{x}_i^k-b\right)\nonumber\\&=&\bar{\lambda}^k-\gamma\left(\tau^k\beta^k\sum_{i=2}^{m}A_i(\widetilde{x}_i^k-\bar{x}_i^k)+(\bar{\lambda}^k-\widetilde{\lambda}^k)\right).\label{G102}
	\end{eqnarray}		
	Since $\bar{x}^{k+1}=\bar{x}^k-\gamma \bar{J}(\bar{x}^k-\widetilde{x}^k)$, it follows from \eqref{G102} that $v^{k+1}=v^k-M^k(v^k-\widetilde{v}^k)$, where $M^k=(P^k)^{-T}N^k.$ Hence, the correction step holds.
	
	We write $H^k$ and $G^k$ satisfying the convergence conditions \eqref{V9}-\eqref{V10} as
	\begin{equation}
	\begin{aligned}
	H^k&=Q^k(N^k)^{-1}(P^k)^T=\frac{1}{\gamma}P^k(P^k)^T=
	\frac{1}{\gamma}\begin{pmatrix}
	\tau^k\beta^kJJ^T&0\\
	0&\frac{1}{\tau^k\beta^k}I_l
	\end{pmatrix},\\
	G^k&=(Q^k)^T+Q^k-(M^k)^TH^kM^k
	\overset{(\ast)}{\succeq} (\frac{1}{\gamma^2}-\frac{1}{\gamma})(N^k)^TN^k,
	\end{aligned}
	\end{equation}	
	where ($\ast$) follows from
	\begin{equation}
	\nonumber
	\begin{aligned}
	&(Q^k)^T+Q^k=\begin{pmatrix}
	\beta^k(J^T+J)&-\widetilde{I}^T\\
	-\widetilde{I}&\frac{2}{\beta^k}I_l
	\end{pmatrix}
	\succeq \begin{pmatrix}
	\beta^k(J^T+J)&-\widetilde{I}^T\\
	-\widetilde{I}&\frac{1}{\beta^k}I_l
	\end{pmatrix}=\frac{1}{\gamma^2}(N^k)^TN^k,\\
	&(M^k)^TH^kM^k=\frac{1}{\gamma}(N^k)^TN^k.
	\end{aligned}
	\end{equation}				
	According to the definition of $N^k$, we have
	\begin{equation}\label{G99}
	\|v^k-\widetilde{v}^k\|^2_{G^k}\geq(1-\gamma)\tau^k\beta^k\|A\widetilde{x}^k-b\|^2.
	\end{equation}
	Based on $\tau^k\beta^k=\beta$, the structure of $H^k$ and \eqref{G99}, the condition \eqref{V12} holds.	
	
	\noindent	\textnormal{(\rmnum{2})}	
	According to the optimality condition of the $x_m$-subproblem in Algorithm \ref{G63} and the assumption that $f_m$ is $L$-gradient Lipschitz continuous, we have
	\begin{equation}
	\begin{aligned} &\frac{1}{(\tau^k)^2}[f_m(x_m)-f_m(\breve{x}_m^{k})]-\frac{1}{(\tau^{k-1})^2}[f_m(x_m)-f_m(\breve{x}_m^{k-1})]+\frac{1}{\tau^k}(x_m-\widetilde{x}_m^k)^T\\&\{-A_m^T\widetilde{\lambda}^k+\tau^k\beta^k\sum_{i=2}^{m}A_m^TA_i(\widetilde{x}_i^k-\bar{x}_i^{k})+(1-\gamma)(1-\tau^k)\beta^kA_m^T(A\breve{x}^{k-1}-b)\}\\&\geq
	\frac{1}{2L\tau^k}\|\nabla f_m(\breve{x}_m^k)-\nabla f_m({x}_m)\|^2+\frac{1-\tau^k}{2L(\tau^k)^2}\|\nabla f_m(\breve{x}_m^k)-\nabla f_m(\breve{x}^{k-1}_m)\|^2\\&\geq
	\frac{1}{2L\tau^k}\|\nabla f_m(\breve{x}_m^k)-\nabla f_m({x}_m)\|^2
	,~\forall x_m.
	\end{aligned}
	\end{equation}
	By an approach  similar to getting \eqref{G82},  we have
	\begin{eqnarray}
	&&\frac{1}{(\tau^k)^2}[f(x)-f(\breve{x}^{k})]-
	\frac{1}{(\tau^{k-1})^2}[f(x)-f(\breve{x}^{k-1})]\nonumber\\
	&&~~+\frac{1}{\tau^k}(u-\widetilde{u}^k)^TF(\widetilde{u}^k)
	+c_2^k(A(x-\widetilde{x}^k))^T(A\breve{x}^{k-1}-b)\label{G119}\\
	\geq&&\frac{1}{\tau^k}
	\left((v-\widetilde{v}^k)^TQ^k(v^k-\widetilde{v}^k)
	+	\frac{1}{2L}\|\nabla f_m(\breve{x}_m^k)-\nabla f_m({x}_m)\|^2\right), ~\forall u,\nonumber
	\end{eqnarray}
	where $c_2^k= (1-\gamma)(1-\tau^k)\beta^k/\tau^k $ and  $Q^k$ is defined in \eqref{G73}. Hence, the prediction step holds.
	It follows from \eqref{G102} and the update of $\bar{x}^{k+1}$ that $v^{k+1}=v^k-M^k(v^k-\widetilde{v}^k)$.  Hence, the correction step holds.
	
	The optimality condition of the $x_m$-subproblem in Algorithm \ref{G63} also implies that
	\begin{equation}\label{G75}
	\begin{aligned}\nabla f_m(\breve{x}_m^k)&
	=A^T_m\lambda^k-\beta^kA^T_m(A\breve{x}^k-b  )\\&
	=A^T_m\bar{\lambda}^k-\tau^k\beta^kA_m^T(A\widetilde{x}^k-b)
	-c_2^k\tau^kA_m^T(A\breve{x}^{k-1}-b)\\
	&\overset{\eqref{G96}}{=}A^T_m
	\left((1-\frac{1}{\gamma})\bar{\lambda}^k+\frac{1}{\gamma}\bar{\lambda}^{k+1}
	-c_2^k\tau^k(A\breve{x}^{k-1}-b)\right).
	\end{aligned}
	\end{equation}
	Since $0= \nabla f_m(x_m^*)-A_m^T\lambda^*$, it follows from \eqref{G75} that
	\begin{equation}\label{G83}
	\begin{aligned}
	&\frac{1}{L}\|\nabla f_m(\breve{x}_m^k)-\nabla f_m({x}^*_m)\|^2\\
	=&\frac{1}{L}\left\|A_m^T\left(\Big(1-\frac{1}{\gamma}\Big)\bar{\lambda}^k+\frac{1}{\gamma}\bar{\lambda}^{k+1}	-\lambda^*-c_2^k\tau^k(A\breve{x}^{k-1}-b)\right)\right\|^2\\ \geq&\sigma'\left\|\left(1-\frac{1}{\gamma}\right)\bar{\lambda}^k+\frac{1}{\gamma}\bar{\lambda}^{k+1}
	-\lambda^*-c_2^k\tau^k(A\breve{x}^{k-1}-b)  \right\|^2\\ \geq&\left(\sigma'-\frac{\sigma'^2}{\sigma'+1}\right)
	\left\|\left(1-\frac{1}{\gamma}\right)\bar{\lambda}+\frac{1}{\gamma}\bar{\lambda}^{k+1}
	-\lambda^*\right\|^2
	-(c_2^k\tau^k)^2\|A\breve{x}^{k-1}-b \|^2,
	\end{aligned}
	\end{equation}
	where the last inequality holds since
	$$\begin{aligned}		\sigma'\|s-t\|^2&=\sigma'\|s\|^2+\sigma'\|t\|^2-2\sigma's^Tt\\
	&\geq\sigma'\|s\|^2+\sigma'\|t\|^2-(\sigma'+1)\|t\|^2-\frac{\sigma'^2}{\sigma'+1}\|s\|^2, ~\forall s,~t\in\mathbb{R}^{l}.
	\end{aligned}
	$$
	We can observe that
	\begin{equation}\label{G76}
	\begin{aligned}
	&\left\|\left(1-\frac{1}{\gamma}\right)\bar{\lambda}^k+\frac{1}{\gamma}\bar{\lambda}^{k+1}
	-\lambda^*\right\|^2\\
	=&
	\frac{1}{\gamma}\|\bar{\lambda}^{k+1}-\lambda^*\|^2-\left(\frac{1}{\gamma}-1\Big)\|\bar{\lambda}^k-\lambda^*\|^2+\frac{1}{\gamma}\Big(\frac{1}{\gamma}-1\right)\|\bar{\lambda}^{k+1}-\bar{\lambda}^k\|^2.
	\end{aligned}
	\end{equation}
	Let $\beta^k=\beta$. Combining  \eqref{G99}, \eqref{G83} and \eqref{G76} yields that
	\begin{small}
		\begin{equation}\nonumber
		\begin{aligned}
		&\frac{1}{\tau^k}\left((v^*-\widetilde{v}^k)^TQ^k(v^k-\widetilde{v}^k)
		+	\frac{1}{2L}\|\nabla f_m(\breve{x}_m^k)-\nabla f_m({x}_m^*)\|^2\right)\\ \geq
		&\frac{1}{2\tau^k}\left(\|v^{k+1}-v^*\|^2_{H^k}+\frac{1}{L}\|\nabla f_m(\breve{x}_m^k)-\nabla f_m({x}^*_m)\|^2-\|v^k-v^*\|^2_{H^k}+\|v^k-\widetilde{v}^k\|^2_{G^k}
		\right)
		\\\geq&\frac{1}{2\tau^k}\left(\|v^{k+1}-v^*\|^2_{H^k}+\frac{\sigma''}{\gamma}\|\bar{\lambda}^{k+1}-\lambda^*\|^2-\|v^k-v^*\|^2_{H^k}-\sigma''(\frac{1}{\gamma}-1)\|\bar{\lambda}^k-\lambda^*\|^2\right)\\&-\frac{(c_2^k\tau^k)^2}{2\tau^k}\|A\breve{x}^{k-1}-b \|^2+\frac{(1-\gamma)\beta}{2}\|A\widetilde{x}^k-b\|^2
		\\ \geq&\frac{1}{2}\left(
		\|v^{k+1}-v^*\|^2_{H_{0}^{k+1}}-\|v^k-v^*\|^2_{H_{0}^k}-(1-\gamma)\beta\left(\frac{\tau^k}{(\tau^{k-1})^2}\|A\breve{x}^{k-1}-b \|^2
		-\|A\widetilde{x}^k-b\|^2\right)\right),
		\end{aligned}
		\end{equation}
	\end{small}
	where the last inequality follows from the condition \eqref{G79} and the fact
	\begin{small}
		\begin{eqnarray*}
			\frac{(c_2^k\tau^k)^2}{\tau^k}=\frac{(1-\gamma)^2\beta^2(1-\tau^k)^2}{\tau^k}\overset{\{(1-\gamma)\beta\leq1\}}{\leq}(1-\gamma)\beta\frac{(\tau^k)^3}{(\tau^{k-1})^4}\leq(1-\gamma)\beta\frac{\tau^k}{(\tau^{k-1})^2}.
		\end{eqnarray*}
	\end{small}
	The proof is complete. \end{proof}

\begin{proof}({\bfseries Proof of Theorem  \ref{T4.8}})	
	Since $\beta^k=\beta/\tau^k$, we have  $\lambda^{k+1}=\lambda^k$ for all $k$ in Algorithm \ref{G63}.
	The following proof is based on the optimality conditions and hence omitted.
\end{proof}

\begin{proof}
	({\bfseries Proof of Theorem  \ref{T4.7}})	
	
	\noindent 	Based on the optimality conditions, similar to getting \eqref{G82} and \eqref{G119}, we can verify the prediction step. Next, according to
	\begin{eqnarray}
	\lambda^{k+1}&=&{\lambda}^k-\gamma\tau^k\beta^k(A\widetilde{x}^{k}-b)\\&=&{\lambda}^k-\gamma\tau^k\beta^k\left(\sum_{i=2}^{m}A_i(\widetilde{x}_i^k-\bar{x}_i^k)+A_1\widetilde{x}_1^k+\sum_{i=2}^{m}A_i\bar{x}_i^k-b\right)\nonumber\\&=&{\lambda}^k-\gamma\left(\tau^k\beta^k\sum_{i=2}^{m}A_i(\widetilde{x}_i^k-\bar{x}_i^k)+({\lambda}^k-\widetilde{\lambda}^k)\right),
	\end{eqnarray}
	we can verify the correction step. The remaining proof of verifying the conditions \eqref{V9}-\eqref{V10} and \eqref{V12} is similar to those given in the proof of Theorem \ref{T4.6}.
\end{proof}

%
%

\bibliographystyle{siamplain}
\bibliography{ref}

\end{document}